\documentclass[
  a4paper, 
  reqno, 
  oneside, 
  11pt
]{amsart}

\usepackage[utf8]{inputenc}

\usepackage[dvipsnames]{xcolor} 
\colorlet{cite}{LimeGreen!50!Green}
\usepackage{tikz}  
\usetikzlibrary{arrows,positioning}
\tikzset{ 
  baseline=-2.3pt,
  text height=1.5ex, text depth=0.25ex,
  >=stealth,
  node distance=2cm,
  mid/.style={fill=white,inner sep=2.5pt},
}

\usepackage{lmodern}
\usepackage{tikz-cd}
\usepackage{amsthm, amssymb, amsfonts}

\usepackage{amsthm, amssymb, amsfonts}	
\usepackage[all]{xy}
\usepackage{graphicx,caption,subcaption}
\usepackage{braket}  
\usepackage{microtype}
\theoremstyle{plain}
\newtheorem{theorem}{Theorem}[section]
\newtheorem{lemma}[theorem]{Lemma}
\newtheorem{proposition}[theorem]{Proposition}

\theoremstyle{definition}

\newtheorem{example}[theorem]{Example}

\newtheorem{remark}[theorem]{Remark}

\usepackage[%
  bookmarks=true,			
  unicode=true,			
  pdftitle={Notes on flat pseudo-Riemannian manifolds},		%
  pdfauthor={Valencia},	%
  pdfkeywords={Flat affine strcture}{Flat pseudo-metrics}{Flat pseudo-Riemannian Lie groups}{Orthogonal Lie groups}{Orthogonal Lie algebras},	
  colorlinks=true,		
  linkcolor=Red,			
  citecolor=cite,		
  filecolor=magenta,		
  urlcolor=RoyalBlue			
]{hyperref}				
\usepackage{cleveref}

\newtheoremstyle{mydef}
  {}		
  {}		
  {}		
  {}		
  {\scshape}	
  {. }		
  { }		
  {\thmname{#1}\thmnumber{ #2}\thmnote{ #3}}

 \DeclareMathAlphabet{\mathpzc}{OT1}{pzc}{m}{it}
 \usepackage{tikz}

 \numberwithin{equation}{section}
 
 \author{ Fabricio Valencia}
 \address{Fabricio Valencia - Instituto de Matem\'aticas, Universidad de Antioquia, Medell\'in, Colombia. E-mail:
 fabricioyarro@gmail.com}
 \title{Notes on flat pseudo-Riemannian manifolds}
 
 \date{\today}

\begin{document}
 	\maketitle
 	\begin{abstract}
 		In these notes we survey basic concepts of affine geometry and their interaction with Riemannian geometry. We give a characterization of affine
 		manifolds which has as counterpart those pseudo-Riemannian manifolds whose Levi-Civita connection is flat. We show that no connected
 		semisimple Lie group admits a left invariant 
flat affine connection. We
 		characterize 
flat pseudo-Riemannian Lie groups. For a 
flat left-invariant
 		pseudo-metric on a Lie group, we show the equivalence between the completeness of the Levi-Civita connection and unimodularity of the group. We emphasize the case of 
flat left invariant hyperbolic metrics on the
 		cotangent bundle of a simply connected 
flat affine Lie group. We also
 		discuss Lie groups with bi-invariant pseudo-metrics and the construction
 		of orthogonal Lie algebras.
 	\end{abstract}
\tableofcontents
\section{Introduction}
A real smooth manifold $M$ of dimension $n$ is called an
{\bf affine manifold} if it admits a maximal atlas whose change of coordinates 
are restrictions of affine transformations of $\mathbb{R}^n$.
Having an affine structure over $M$ is equivalent to having a flat and torsion free linear connection $\nabla$ on $TM$
(see Theorem \ref{Characterization1}).
A pair $(M,\nabla)$, where $\nabla$ is a flat affine connection
(i.e. $\nabla$ is a flat and torsion free linear connection) on $M$,
is called a {\bf flat affine manifold}.
When $M=G$ is a Lie group and $\nabla$ is a
left invariant flat affine connection,
the pair $(G,\nabla)$ is called a {\bf flat affine Lie group}.
If $g$ is a pseudo-metric on $M$
(respectively $\mu$ is a left invariant pseudo-metric on $G$)
such that the Levi-Civita connection associated to $g$ has vanishing  curvature tensor,
the  pair $(M,g)$ (respectively $(G,\mu)$) is called a 
flat pseudo-Riemannian manifold (respectively flat pseudo-Riemannian Lie group).

These notes are organized as follows.
The first two sections are devoted to the study of flat affine manifolds.
Theorem \ref{characterizatioFALIC} is essential because
it gives a characterization of flat affine Lie groups that we will use throughout
these notes.
We show that no connected semisimple real Lie group admits a
left invariant flat affine connection (Theorem \ref{Semisimplenoaffine}).
In Section 3 we introduce some basic concepts of Riemannian geometry
and exhibit  examples of flat affine structures compatible with pseudo-metrics.
Section 4 is dedicated to the study of flat pseudo-Riemannian Lie groups.
We give a characterization of such Lie groups and we show that 
the left-invariant affine structure defined by the Levi-Civita connection is
geodesically complete if and only if
the group is unimodular (Theorem \ref{unimodularcompletes}).
We also show that the cotangent bundle of a simply connected
flat affine Lie group is endowed with an affine
Lie group structure and a left invariant flat hyperbolic metric
(Proposition \ref{AuMcotangent}).
In the Section 7 we study {\bf orthogonal Lie groups},
that is, Lie groups endowed with bi-invariant metrics.
To study properties of orthogonal Lie groups we  introduce
the notion of {\bf orthogonal Lie algebra}, which
will be used in the method of {\bf double orthogonal extension}.
As an application, we  describe how to construct the oscillator Lie algebra of the oscillator Lie group which
appear in several branches of Physics and Mathematical-Physics
and give rise to particular solutions of the Einstein-Yang-Mills equations.
Finally, we present another characterization of  flat  Riemannian Lie
groups using some consequences of the presence of an orthogonal structure in a Lie algebra (Theorem \ref{MilnorflatTheorem}).

\section{Flat affine manifolds}
In what follows  $M$ will denote a  connected paracompact real smooth manifold of dimension  $n$. 
We will denote by  $\mathfrak{X}(M)$ the Lie algebra of smooth vector fields over $M$ and by $C^{\infty}(M)$ 
the associative algebra of functions on $M$ with values in $\mathbb{R}$.

The objects of study of these notes are flat affine paracompact manifolds. 
In particular, we study flat affine structures that are compatible with pseudo-Riemannian metrics. 
A good understanding of the category of Lagrangian submanifolds requires 
a good knowledge of the category of flat affine manifolds (see \cite[Thm 7.8]{W}). 
Also,  flat affine manifolds with holonomy reduced to $\text{Gl}(n,\mathbb{Z})$ appear naturally in integrable
systems and Mirror symmetry (see \cite{KS}). 
Further applications of  flat affine manifolds appear in the study of 
Hessian structures and Information Geometry (see \cite[c.\thinspace 6]{Sh}).

Let $V$ be a real finite dimensional vector space. 
The space of {\bf affine transformations of $V$} is the  
Lie group $\text{Aff}(V)=V\rtimes_{Id} \text{GL}(V)$ determined by 
the semi-direct product of the Abelian Lie group  $(V,+)$ with the  Lie group  $\text{GL}(V)$   
via the identity representation. 
Its Lie algebra is the product vector space $\mathfrak{aff}(V)=V\rtimes_{id}\mathfrak{gl}(V)$ with   
Lie bracket given by 
$$[(x,t),(y,s)]=(t(y)-s(x),[t,s]_{\mathfrak{gl}(V)}),\quad$$ 
for all $x,y\in V$ and $t,s\in\mathfrak{gl}(V)$.

We say that $M$ {\bf admits an affine structure} 
if there exists an maximal atlas $\lbrace (U_\alpha,\varphi_\alpha) \rbrace_{\alpha\in J}$ 
of $M$ having change of coordinates  that are restrictions of affine transformations of  
$\mathbb{R}^n$, that is, for each $\alpha,\beta\in J$ with $U_\alpha\cap U_\beta\neq \emptyset$, 
there exists $\sigma_{\alpha\beta}\in\text{Aff}(\mathbb{R}^n)$ such that
$$\left.\varphi_\beta\circ \varphi_\alpha^{-1}\right|_{\varphi_\alpha(U_\alpha\cap U_\beta)}= 
\left.\sigma_{\alpha\beta}\right|_{\varphi_\alpha(U_\alpha\cap U_\beta)}.$$

If $G$ is a discrete  Lie subgroup of  $\text{Aff}(\mathbb{R}^n)$ 
that acts freely and  properly discontinuously over $\mathbb{R}^n$, 
then the quotient manifold  $\mathbb{R}^n/G$ admits an affine structure 
such that the coordinate changes  are restrictions of elements of   $G$ (see \cite[p. 349]{R}).

\begin{example} Let $S^1=\{z\in \mathbb{C}: |z|=1\}$, $U_1=S^1-\lbrace (1,0)
	\rbrace$ and $U_2=S^1-\lbrace (0,1)
	\rbrace$. If $\varphi_1: U_1 \to (0,2\pi)$ and $\varphi_2: U_2 \to \left( -\dfrac{\pi}{2},3\dfrac{\pi}{2}\right)$ 
	are defined respectively by
	$$z \mapsto \text{arg}(z)\quad\text{and}\quad z \mapsto \left\{ \begin{array}{lcc}
	\text{arg}(z)-\dfrac{\pi}{2},&   \text{if}  & \text{arg}(z)\in \left( \dfrac{\pi}{2},2\pi\right) \\
	\\ \text{arg}(z)+3\dfrac{\pi}{2},&  \text{if} & \text{arg}(z)\in \left( 0,\dfrac{\pi}{2}\right)
	\end{array},
	\right.$$
	then the atlas $\lbrace (U_1,\varphi_1), (U_2,\varphi_2) \rbrace$ 
	determines an affine structure for $S^1$.
\end{example}
\begin{example}
	\emph{Hopf manifolds}. Let $\lambda>1$ be a fixed real number. 
	Denote by  $G$ the group of transformations of  
	$\mathbb{R}^n\backslash\lbrace 0\rbrace$ defined by
	\begin{align*}
	T_n: \mathbb{R}^n\backslash\lbrace 0\rbrace &\to \mathbb{R}^n\backslash\lbrace 0\rbrace\\
	x &\mapsto T_n(x)=\lambda^n\cdot x,
	\end{align*}
	for all  $n\in \mathbb{Z}$. 
	The set $G$ is a discrete subgroup of  $\text{Aff}(\mathbb{R}^n)$ 
	that acts freely and properly discontinuously  over 
	$\mathbb{R}^n\backslash\lbrace 0\rbrace$. Therefore $\mathbb{R}^n\backslash\lbrace 0\rbrace/G$ 
	is an  affine manifold called a {\bf Hopf manifold} which we will denote by  $\mathrm{Hopf}(\lambda,n)$. 
	Topologically these manifolds are either the disjoint union of 
	two Hopf circles $\mathbb{R}^+/G$, when $n=1$, or diffeomorphic to  $S^{n-1}\times S^1$ when $n>1$.
\end{example}

Recall that a  {\bf linear connection} on a smooth manifold $M$ is an  
$\mathbb{R}$-bilinear map 
$\nabla: \mathfrak{X}(M)\times \mathfrak{X}(M) \to \mathfrak{X}(M)$ 
that is  $C^\infty(M)$-linear on the first component and satisfies 
$$\nabla_XfY=X(f)Y+f\nabla_XY,$$ 
for all $X,Y\in\mathfrak{X}(M)$ and $f\in C^{\infty}(M)$.
The {\bf torsion tensor}  $T_\nabla$  and  {\bf curvature tensor}  $R_\nabla$ associated to a linear connection  $\nabla$ are defined respectively by
\begin{equation}\label{torsion}
T_\nabla(X,Y)=\nabla_XY-\nabla_YX-[X,Y]
\end{equation}
and
\begin{equation}\label{curvature}
R_\nabla(X,Y)Z=\nabla_X\nabla_YZ-\nabla_Y\nabla_XZ-\nabla_{[X,Y]}Z,
\end{equation}
for all $X,Y,Z \in\mathfrak{X}(M)$. When  $T_\nabla=0$ and $R_\nabla= 0$ we say that $\nabla$ is a  
{\bf flat affine connection} and the pair $(M,\nabla)$ is called a {\bf flat affine manifold}.
\begin{remark}
	The pair $(M,\nabla)$ is a flat affine manifold  if and only if there exists an  
	atlas for $M$ such that the  Christoffel symbols
	associated to  $\nabla$ vanish identically on all charts  (see \cite[p.\thinspace 108]{C}).
\end{remark}

\begin{example} If $(x^1,\cdots,x^n)$  are the usual coordinates in 
	$\mathbb{R}^n$, the {\bf usual linear connection} $\nabla^0$ on $\mathbb{R}^n$ 
	is defined as
	$$\nabla^0_XY=\sum_{j=1}^{n}X(f^j)\dfrac{\partial}{\partial x^j},\quad\text{where}\quad Y=\sum_{j=1}f^j\dfrac{\partial}{\partial x^j}.$$
	It is simple to verify that $\nabla^0$ is a flat affine connection on 
	$\mathbb{R}^n$ and that we have $\Gamma_{ij}^k=0$ for all $i,j,k=1,\cdots,n$.
\end{example}

From now on, by smooth manifolds we mean real manifolds that are $C^\infty$ differentiable. 
The following characterization of affine manifolds appears in  \cite{AM}.

\begin{theorem}[Auslander-Markus]\label{Characterization1}
	A real smooth manifold $M$ has an affine structure if and only if  
	there exists a flat affine connection on $M$.
\end{theorem}

\begin{proof}
	Suppose that $\lbrace (U_\alpha,\varphi_\alpha) \rbrace_{\alpha\in J}$ 
	is an affine structure for $M$. 
	For each $\alpha\in J$, we endow the open set  $\varphi_\alpha(U_\alpha)\subseteq \mathbb{R}^n$ with the usual linear connection $\nabla^0$. 
	The pullback of $\left.\nabla^0\right|_{\varphi_\alpha(U_\alpha)}$ 
	by the diffeomorphism $\varphi_\alpha$ defines a flat affine connection $\nabla_\alpha$ over $U_\alpha$.
	We choose  $\nabla$ over $M$ as the linear connection subjected to
	$\left.\nabla\right|_{U_\alpha}=\nabla_\alpha$ for all $\alpha\in J$. 
	To verify that $\nabla$ is well defined observe that, for 
	$\alpha,\beta\in J$ with $U_\alpha\cap U_\beta\neq \emptyset$, 
	setting $\varphi_\alpha=(x^1,\cdots,x^n)$ and $\varphi_\beta=(y^1,\cdots,y^n)$
	on  $U_\alpha\cap U_\beta$, we have
	\begin{equation}\label{Christoffelsymbols}
	(\Gamma_{jk}^i)_\alpha=\sum_{l=1}^n\dfrac{\partial^2 y^l}{\partial x^j\partial x^k}\dfrac{\partial x^i}{\partial y^l}+\sum_{l,m,q=1}^n(\Gamma_{mq}^l)_\beta\dfrac{\partial y^m}{\partial x^j}\dfrac{\partial y^q}{\partial x^k}\dfrac{\partial x^i}{\partial y^l}.
	\end{equation}
	Given that the  Christoffel symbols of  $\nabla^0$ vanish, 
	we obtain $(\Gamma_{mq}^l)_\beta=0$ on $U_\beta$. 
	Moreover we have $\dfrac{\partial^2 y^l}{\partial x^j\partial x^k}=0$, 
	since $\left.\varphi_\beta\circ \varphi_\alpha^{-1}\right|_{\varphi_\alpha(U_\alpha\cap U_\beta)}$
	is the restriction of an element of  
	$\text{Aff}(\mathbb{R}^n)$, and hence we obtain $(\Gamma_{jk}^i)_\alpha=0$, 
	showing that 
	$\nabla$ is well defined. 
	Furthermore, using equation \eqref{Christoffelsymbols} we can verify that  
	$\nabla$ is the unique flat affine connection that can be obtained in this fashion.
	
	Reciprocally, suppose that $\nabla$ is a flat affine connection on $M$. 
	For each $p\in M$ there exists a neighborhood $V_p$ of 0 in $T_p M$ and a neighborhood  
	$U_p$ of $p$ in $M$ such that the exponential map associated to $\nabla$, 
	denoted by $\exp_p\colon V_p\to U_p$,
	is a diffeomorphism (see \cite[p.\thinspace 148]{KN}). 
	Given a basis $\lbrace X_1,\cdots,X_n \rbrace$ for the tangent space $T_p M$, 
	we define local charts on $U_p$ by
	$$x^i\left(\exp_p\left(\sum_{j=1}^n a^jX_j \right)\right)=a^i,$$
	if $\sum_{j=1}^n a^jX_j\in V_p$ for all $ i=1,\cdots, n$.
	Since $\nabla$ is flat affine, there exists an atlas over $M$ with respect to which we have 
	$\Gamma_{ij}^{k}=0$ for every chart. 
	The computation of geodesic curves $\gamma$ in a chart $(U,(y^1,\cdots,y^n))$ 
	of such an atlas amounts to solving the system of ordinary differential equations  
	$\dfrac{d^2y^i(\gamma(t))}{dt^2}=0$ for $i=1,\cdots,n$, 
	whose solution, for a fixed initial condition $\left(p,\sum_{j=1}^n a^jX_j\right)\in TM$, is unique. 
	Therefore, setting $\dfrac{\partial^2 x*^l}{\partial x^j\partial x^k}=0$ 
	on each intersection, these normal coordinates $(U_p,(x^1,\cdots,x^n))$ 
	generate a unique  atlas over $M$ for which the changes of coordinates are 
	restrictions of elements of  $\text{Aff}(\mathbb{R}^{n})$.
\end{proof}

In general, determining whether a smooth manifold admits 
a flat affine structure or not is a difficult question, 
and there are obstructions for the existence of said structures. 

\begin{example} We list some manifolds that do not admit flat affine structures:
	\begin{itemize}
		\item Compact simply connected manifolds (see \cite{Eh}). 
		\item Compact manifolds with finite fundamental group (see \cite{AM}). 
		\item In particular for  $n>1$ the real $n$-sphere $S^n$, the real projective space  $\mathbb{RP}^n$ 
		and the group of rotations $O(n)^+$ do not admit flat affine structures. 
	\end{itemize}
\end{example}
Further topological obstructions for the existence of a flat affine structures are listed in \cite{Sm}.

\begin{remark}
	There is no direct relation between the notion of  
	affine variety as given in algebraic geometry (namely a set cut out  by polynomial equations) 
	and the definition of affine manifold in the way that we present it (when a manifold admits an affine structure). 
	For example, for $n>1$ the  $n$-dimensional real sphere $S^n$ 
	is an affine algebraic variety but is not a flat affine manifold in the sense of our definition.
\end{remark}

\section{Flat affine Lie groups}
In what follows $G$ denotes a connected real Lie group. For each $\sigma\in G$, we denote by $L_\sigma:G\to G$ the map left multiplication by $\sigma$ in $G$, that is, the map defined by $\tau\mapsto L_\sigma(\tau)=\sigma\tau$. The tangent space $T_\epsilon G$ of $G$ at the identity and the Lie algebra of left invariant vector fields $\mathfrak{X}_{l}(G)$ on $G$ are isomorphic vector spaces as follows. For each $x\in T_\epsilon G$, we associate the left invariant vector field $x^+$ defined by
$$x^+_\sigma=(L_\sigma)_{\ast,\epsilon}(x)=\left.\dfrac{d}{dt}\right|_{t=0}(\sigma\cdot \text{exp}_G(tx)),$$
for all $\sigma\in G$. Under this isomorphism we give a structure of Lie algebra to $\mathfrak{g}=T_\epsilon G$ and call it the Lie algebra of $G$.\\

A linear connection  $\nabla$ on $G$ is called {\bf left invariant} if $L_\sigma$ is an affine transformation of $(G,\nabla)$ for all $\sigma\in G$. More precisely, we must have
$$(L_{\sigma^{-1}})_\ast\left(\nabla_{(L_\sigma)_\ast X}(L_\sigma)_\ast Y\right)=\nabla_XY,$$
for all $X,Y\in\mathfrak{X}(G)$ and $\sigma\in G$. From this definition it follows immediately that a connection 
$\nabla$ on $G$ is left invariant if and only if for all 
$x^+,y^+\in \mathfrak{X}_{l}(G)$ we have $\nabla_{x^+}y^+\in\mathfrak{X}_{l}(G)$.

\begin{lemma}\label{Leftinvariantconnection1}
	There exists a bijective correspondence between left invariant linear connections  
	on $G$ and bilinear maps on $\mathfrak{g}$.
\end{lemma}
\begin{proof}
	If $\nabla$ is a left invariant linear connection on $G$, 
	then the assignment $\cdot: \mathfrak{g}\times\mathfrak{g} \to \mathfrak{g}$, 
	given by $(x,y) \mapsto x\cdot y=(\nabla_{x^+}y^+)(\epsilon)$ for all  
	$x,y\in \mathfrak{g}$, defines a bilinear map on  $\mathfrak{g}$. 
	Conversely, suppose that $\cdot:\mathfrak{g}\times\mathfrak{g}\to \mathfrak{g}$ 
	is a bilinear map on $\mathfrak{g}$. Define $\nabla$ on $G$ 
	as the linear connection such that $\nabla_{x^+}y^+=(x\cdot y)^+$
	verifies
	\begin{equation}\label{leftinv1}
	\nabla_{fx^+}y^+=f.(x\cdot y)^+ \qquad \text{and} \qquad\nabla_{x^+}fy^+=x^+(f)y^++f.(x\cdot y)^+,
	\end{equation}
	for all $x,y\in\mathfrak{g}$ and $f\in C^\infty(G)$.
	Since the left invariant vector fields determine an absolute parallelism 
	over $G$, we have that $\mathfrak{X}_{l}(G)$ generates  
	$\mathfrak{X}(G)$ as a $C^\infty(G)$-module. 
	Hence, each smooth vector field over  $G$ can be written as a  $C^\infty(G)$-linear
	combination of  left invariant vector fields. 
	Using this fact together with the identities exhibited in  \eqref{leftinv1} 
	we can easily conclude that  $\nabla$ is a left invariant linear connection on  $G$.
\end{proof}
When there exists a left invariant flat affine connection $\nabla$ on   
$G$, the pair $(G,\nabla)$ is called  a {\bf flat affine  Lie group}. To characterize flat affine  Lie groups and to study their 
structure is an open problem which was proposed by   J. Milnor in \cite{M}. 
The following characterization of flat affine  Lie groups was given in  \cite{K} and \cite{Me}.

\begin{theorem}[Koszul and Medina]\label{characterizatioFALIC}
	Let $G$ be a connected  $n$-dimensional real Lie group, 
	$\mathfrak{g}$ its  Lie algebra and $\widetilde{G}$ its universal covering Lie group. 
	Then, the following are equivalent.
	\begin{enumerate}
		\item There exists a left invariant flat affine connection on $G$.
		\item There exists a bilinear map $\cdot:\mathfrak{g}\times\mathfrak{g}\to \mathfrak{g}$ on $\mathfrak{g}$ such that
		\begin{equation}\label{compatiblebracket}
		[x,y]=x\cdot y-y\cdot x
		\end{equation}
		and
		\begin{equation}\label{leftsymmetricproduct}
		L_{[x,y]}=[L_x,L_y]_{\mathfrak{gl}(\mathfrak{g})},
		\end{equation}
		for all $x,y\in\mathfrak{g}$, here $L_x:\mathfrak{g}\to\mathfrak{g}$ is the map defined by 
		$y\mapsto L_x(y)=x\cdot y$. 
		\item There exists a real $n$-dimensional vector space $V$ and a  
		Lie group homomorphism $\rho\colon \widetilde{G}\to \text{Aff}(V)$ 
		such that the left action of   $\widetilde{G}$ over $V$ defined by 
		$\sigma\cdot v=\rho(\sigma)(v)$ for all $(\sigma,v)\in \widetilde{G}\times V$, 
		allows a point having  open orbit and discrete isotropy.
	\end{enumerate}
\end{theorem}
\begin{proof}
	We first show that $1$ implies $2$. Let $\nabla$ be a left invariant flat affine connection on $G$. 
	By Lemma \ref{Leftinvariantconnection1} we have that  
	$L_x(y)=x\cdot y=(\nabla_{x^+}y^+)(\epsilon)$ defines a bilinear map on $\mathfrak{g}$. 
	Substituting this equality into the formulas of torsion and curvature  
	\eqref{torsion}-\eqref{curvature} for $\nabla$, we obtain identities 
	\eqref{compatiblebracket} and \eqref{leftsymmetricproduct}, respectively. 
	
	To get  $2$ implies $3$, suppose that there exists a bilinear map 
	$\cdot:\mathfrak{g}\times\mathfrak{g}\to \mathfrak{g}$ on 
	$\mathfrak{g}$ satisfying \eqref{compatiblebracket} and \eqref{leftsymmetricproduct}, 
	where $L_x:\mathfrak{g}\to\mathfrak{g}$ is the linear map defined by  
	$y\mapsto L_x(y)=x\cdot y$, for all  $x,y\in\mathfrak{g}$. 
	Then the map  $\theta: \mathfrak{g} \to \mathfrak{aff}(\mathfrak{g})$, 
	defined by $x \mapsto (x,L_x)$, is a well defined  Lie algebra homomorphism. 
	This follows from the fact that \eqref{compatiblebracket} and 
	\eqref{leftsymmetricproduct} imply that the map 
	$L:\mathfrak{g}\to\mathfrak{gl}(\mathfrak{g})$, defined by $x\mapsto L_x$, 
	is a well defined  Lie algebra homomorphism which satisfies
	$[x,y]=L_x(y)-L_y(x)$ for  all $x,y\in\mathfrak{g}$. 
	On the other hand, for $0\in\mathfrak{g}$ the map 
	$\psi_0: \mathfrak{g} \to \mathfrak{g}$ given by $x \mapsto x+L_x(0)=x$ is a linear isomorphism. 
	Thus, by means of the exponential map of $G$, we obtain a homomorphism of  
	Lie groups $\rho: \widetilde{G} \to \text{Aff}(\mathfrak{g})$ given by 
	$\sigma \to (Q(\sigma),F_\sigma)$, where for $\sigma=\exp_G(x)$ we have
	$$Q(\sigma)=\sum_{k=1}^\infty \dfrac{1}{k!}(L_x)^{k-1}(x)\quad\text{and}\quad F_\sigma=\text{Exp}(L_x)=\sum_{k=0}^\infty \dfrac{1}{k!}(L_x)^k.$$
	Since  $\psi_0$ is surjective, the orbit of $0\in \mathfrak{g}$ 
	by the left action of  $\widetilde{G}$ over $\mathfrak{g}$, 
	defined by  $\sigma \cdot 0=Q(\sigma)+ F_\sigma(0)=Q(\sigma)$ for all  
	$\sigma\in\widetilde{G}$, is open. Moreover, by the injectivity 
	of $\psi_0$ it follows that the isotropy of $0\in\mathfrak{g}$ by the given action is dicrete. 
	The latter implies that the orbital map $\pi:\widetilde{G}\to \text{Orb}(0)$, 
	given by  $\sigma\mapsto Q(\sigma)$,  is a local diffeomorphism and hence a covering map (see \cite{Me}).
	
	Finally let us show that $3$ implies $1$. 
	Let $V$ be a real vector space of dimension $n$ and assume that there exists a   
	Lie group homomorphism $\rho:\widetilde{G} \to \text{Aff}(V)$, defined by  
	$\sigma \to (Q(\sigma),F_\sigma)$, 
	which admits a point  $v\in V$ with open orbit and discrete isotropy  
	for the action of $\widetilde{G}$ on the left over $V$ induced by
	$\rho$. The latter implies that the map $F:\widetilde{G}\to \text{GL}(V)$ 
	defined by $\sigma\mapsto F_\sigma$, is a Lie group homomorphism and  
	$Q:\widetilde{G}\to V$, given by  $\sigma\mapsto Q(\sigma)$, is a smooth map that satisfies
	$$Q(\sigma\tau)=Q(\sigma)+F_\sigma(Q(\tau)),$$
	for all $\sigma,\tau\in\widetilde{G}$. Moreover, the orbital map $\pi:\widetilde{G}\to \text{Orb}(v)$ 
	given by $\sigma\mapsto Q(\sigma)+F_\sigma(v)$ is a local diffeomorphism. 
	Differentiating at the identity of  $\widetilde{G}$, we obtain the Lie algebra homomorphism 
	$\theta: \mathfrak{g} \to \mathfrak{aff}(V)$ given by  $x \mapsto (q(x),f_x)$ for all  
	$x\in\mathfrak{g}$, where the map $f:\mathfrak{g}\to \mathfrak{gl}(V)$ defined by  
	$x\mapsto f_x$, is a  Lie algebra homomorphism and   
	$q:\mathfrak{g}\to V$, given by $x\mapsto q(x)$, is the linear map
	\begin{equation}\label{1cocycle}
	q([x,y])=f_x(q(y))-f_y(q(x)),
	\end{equation}
	for all $x,y\in\mathfrak{g}$. Moreover, the map  $\psi_v: \mathfrak{g} \to V$ defined by 
	$x \mapsto q(x)+f_x(v)$ is a linear isomorphism. 
	Now, for each $x\in \mathfrak{g}$ we define
	$$L_x=\psi_v^{-1}\circ f_x\circ \psi_v.$$
	Since $f:\mathfrak{g}\to \mathfrak{gl}(V)$ is a Lie algebra homomorphism, 
	we have $L_{[x,y]}=[L_x,L_y]_{\mathfrak{gl}(\mathfrak{g})}$ 
	for all $x,y\in \mathfrak{g}$. On the other hand, since 
	$q:\mathfrak{g}\to V$ satisfies \eqref{1cocycle}, we conclude 
	$[x,y]=L_x(y)-L_y(x)$ for all $x,y\in \mathfrak{g}$. 
	Therefore, using Lemma \ref{Leftinvariantconnection1}, we obtain that the
	linear connection $\nabla$ defined by
	$$\nabla_{x^+}y^+=(x\cdot y)^+=(L_x(y))^+,$$ 
	for all $x,y\in\mathfrak{g}$, is a left invariant flat affine connection on $G$. 
	Using the linear isomorphism $\psi_v$, it can be easily drawn that 
	the Lie algebra homomorphisms in $\mathfrak{g}\to \mathfrak{aff}(V)$ 
	defined by $x\mapsto (x,L_x)$ and $x\to (q(x),f_x)$ 
	are isomorphic.
\end{proof}
\begin{example}
	\emph{Dimension 2.} Recall that the  Lie group of affine transformations 
	of the real line is given by the product manifold
	$\text{Aff}(\mathbb{R})=\mathbb{R}^*\times\mathbb{R}$, 
	with product $(a,b)\cdot (c,d)=(ac,ad+b)$. 
	Its Lie algebra is identified with the  vector space  
	$\mathfrak{aff}(\mathbb{R})=\text{Vect}_\mathbb{R}\lbrace e_1,e_2 \rbrace$ 
	with Lie  bracket $[e_1,e_2]=e_2$. 
	Next we introduce is a family of left invariant flat affine connections   
	on $\text{Aff}(\mathbb{R})$ which are not isomorphic. For $\alpha$ real, set
	$$\nabla_{e_1^+}e_1^+=\alpha e_1^+, \quad\nabla_{e_1^+}e_2^+=e_2^+,\quad\nabla_{e_2^+}e_1^+=\nabla_{e_2^+}e_2^+=0.$$
	Here $e_1^+=x\dfrac {\partial}{\partial x}$ and $e_2^+=x\dfrac {\partial}{\partial y}$ are the 
	left invariant vector fields associated to  $e_1$ and $e_2$, respectively. 
	A description of left invariant flat affine structures over 
	$\text{Aff}(\mathbb{R})$ can be found in  \cite{MSG}.
\end{example}
\begin{example}
	\emph{Dimension 3.} The {\bf Heisenberg} Lie group of dimension  $3$ 
	is given by the set of matrices
	$$H_3=\left\lbrace\begin{pmatrix}
	1 & x & z\\
	0 & 1 & y\\
	0 & 0 & 1 
	\end{pmatrix}:\ x,y,z\in \mathbb{R}\right\rbrace.$$
	The Lie algebra of $H_3$ is identified with 
	$\mathfrak{h}_3=\text{Vect}_\mathbb{R}\lbrace e_1,e_2,e_3 \rbrace$ with Lie  
	bracket $[e_1,e_2]=e_3$. The following is a left invariant flat affine connection on $H_3$:
	$$\nabla_{e_1^+}=\nabla_{e_3^+}=0, \quad\nabla_{e_2^+}e_1^+=-e_3,\quad\nabla_{e_2^+}e_2^+=e_1^+,\quad\nabla_{e_2^+}e_3^+=0.$$
	The vector fields $e_1^+=\dfrac{\partial}{\partial x}$, 
	$e_2^+=\dfrac{\partial}{\partial y}+x\dfrac{\partial}{\partial z}$, 
	and $e_3^+=\dfrac{\partial}{\partial z}$, denote the left invariant vector fields 
	associated to  $e_1$, $e_2$ and $e_3$, respectively. 
\end{example}

\begin{example}
	\emph{Dimension 4.} The product manifold $\mathbb{R}\ltimes_\rho\mathbb{R}^3$ 
	has the structure of a Lie group given by the  semidirect product of the 
	Abelian Lie group  $(\mathbb{R}^3,+)$ with  $(\mathbb{R},+)$ 
	via the  Lie group homomorphism
	\begin{align*}
	\rho: \mathbb{R} &\longrightarrow  \text{GL}(\mathbb{R}^3)\\
	t &\longmapsto \begin{pmatrix}
	e^t & 0 & 0\\
	0 & e^{-t} & 0 \\
	0 & 0 & 1
	\end{pmatrix}.
	\end{align*} 
	Next we introduce a family of  flat left invariant affine connections on 
	$\mathbb{R}\ltimes_\rho\mathbb{R}^3$:
	$$\nabla_{e_2^+}=\nabla_{e_3^+}=\nabla_{e_4^+}=0, \quad\nabla_{e_1^+}e_1^+=\alpha e_1^+,$$
	$$\nabla_{e_2^+}e_2^+=e_2^+,\quad\nabla_{e_1^+}e_3^+=-e_3^+,\quad\nabla_{e_1^+}e_4^+=0,$$
	for all $\alpha\in\mathbb{R}$. 
	The vector fields 
	$$e_1^+=\dfrac{\partial}{\partial t},\quad e_2^+=e^t\dfrac{\partial}{\partial x},\quad e_3^+=e^{-t}\dfrac{\partial}{\partial y},\quad e_4^+=\dfrac{\partial}{\partial z},$$
	determine a basis for $\mathfrak{X}_{l}(\mathbb{R}\ltimes_\rho\mathbb{R}^3)$.
\end{example}

Recall that a Lie group $G$ is called {\bf semisimple} 
if its Lie algebra decomposes into a direct sum of  simple  Lie algebras.
An interesting result, due to C. Chevalley and S. Eilenberg (see \cite{CE}) 
states that a Lie algebra $\mathfrak{g}$ is semisimple 
if and only if we have $H^1(\mathfrak{g},\theta)=0$
for every real representation $\theta$ of $\mathfrak{g}$  
over a finite dimensional vector space. 
Accordingly, we have the following beautiful result of   A. Bon-Yau Chu in \cite{BY}.

\begin{theorem}[Bon-Yau Chu]\label{Semisimplenoaffine}
	Let $G$ be a real  semisimple Lie group. Then $G$ does not admit a left invariant flat affine connection.
\end{theorem}

\begin{proof} Let $G$ be a semisimple real  Lie group of dimension $n$ and $\mathfrak{g}$ its Lie algebra. 
	Since $\mathfrak{g}$ is semisimple, its derived ideal satisfies 
	$\mathfrak{g}=[\mathfrak{g},\mathfrak{g}]$. 
	This implies that every linear representation  $\theta$ of
	$\mathfrak{g}$ on a finite dimensional  vector space has trace $\text{tr}(\theta(x))=0$ 
	for all $x\in\mathfrak{g}$. 
	Suppose that there exists a left invariant flat affine connection $\nabla$ on $G$. 
	Then, by  Theorem \ref{characterizatioFALIC}, the map $L:\mathfrak{g} \to \mathfrak{gl}(\mathfrak{g})$ 
	defined by $x \mapsto L_x$, where $L_x:\mathfrak{g}\to \mathfrak{g}$ is the linear map given by
	$ y \mapsto L_x(y)=x\cdot y=(\nabla_{x^+}y^+)(\epsilon)$, 
	is a linear representation of $\mathfrak{g}$ on the vector space $\mathfrak{g}$. 
	We denote by $C^p(\mathfrak{g},L)$ and $H^p(\mathfrak{g},L)$ the spaces of  
	$p$-cochains and the $p$-th cohomology group of  $\mathfrak{g}$ 
	associated to the linear representation $L$, respectively. 
	We define  $\gamma\in C^1(\mathfrak{g},L)$ by $\gamma(x)=x$ for all 
	$x\in\mathfrak{g}$. Then, since $\nabla$ is torsion free and left invariant, we have
	$$\text{d}\gamma(x,y)=L_x(\gamma(y))-L_y(\gamma(x))-\gamma([x,y])=x\cdot y-y\cdot x-[x,y]=0,$$
	for all $x,y\in\mathfrak{g}$. 
	Therefore, we have $\text{d}\gamma= 0$. 
	Since $\mathfrak{g}$ is semisimple, we obtain $H^1(\mathfrak{g},L)=0$. 
	Hence, there exists $z\in C^0(\mathfrak{g},L)=\mathfrak{g}$ 
	such that $x=\gamma(x)=\text{d}z(x)=L_x(z)$ for all $x\in\mathfrak{g}$. 
	Once again, since the torsion tensor of $\nabla$ is null we reach
	$$x=L_x(z)=x\cdot z=z\cdot x-[z,x]=(L_z-\text{ad}_z)(x),$$
	which implies the relation $L_z=I+\text{ad}_z$, where $I$ and $\text{ad}$ 
	are the identity map and the adjoint representation of $\mathfrak{g}$, respectively. 
	Since $L$ and $\text{ad}$ are linear representations of $\mathfrak{g}$, we obtain
	$0=\text{tr}(L_z)=\text{tr}(I)+\text{tr}(\text{ad}_z)=\text{dim}(\mathfrak{g})=n,$
	which is a contradiction.
\end{proof}

\begin{example}
	The special  linear group  $\text{SL}(n,\mathbb{R})$, 
	the special orthogonal group $\text{SO}(n,\mathbb{R})$ and 
	the symplectic linear group $\text{Sp}(n,\mathbb{R})$ 
	do not allow a structure of flat affine Lie group, given that they are semisimple.
\end{example}

\section{Flat pseudo-Riemannian manifolds}
Our next objective is to study  left invariant flat affine structures over 
Lie groups in the case when these structures are compatible with a pseudo-Riemannian metric. 
To do so, we introduce the following structures from Riemannian geometry.
Let  $M$ be a smooth connected paracompact manifold of real dimension $n$. 
For each $p\in M$, we denote by 
$L^2(T_p M,\mathbb{R})$ the set of all bilinear maps $\beta:T_p M\times T_p M\to \mathbb{R}$. 
Recall that the {\bf index} $\nu$ of a symmetric bilinear form $\beta$ 
on a real finite-dimensional vector space $V$ is the largest integer
that is the dimension of a subspace $W\subset V$ on which $\beta|_W$ is negative definite. 
Equivalently, if $\beta$ is also non-degenerate, the index $\nu$ of $V$ 
is the number of $-1$ in the diagonal of the matrix representation of 
$\beta$ with respect to any orthonormal basis of $V$.

A {\bf pseudo-metric} $g$ on $M$ is an assignment $p\mapsto g_p\in L^2(T_p M,\mathbb{R})$ 
such that the following conditions are met:
\begin{enumerate}
	\item $g_p(X_p,Y_p)=g_p(Y_p,X_p)$ for all $X_p,Y_p\in T_p M$,
	\item $g_p$ is non-degenerate for all $p\in M$,
	\item if $(U,(x^1,\cdots,x^n))$ is a chart of  $M$, the coefficients $g_{ij}$ of the  local representation
	$$g_p=\sum_{i,j=1}^n g_{ij}(p)\cdot dx^i|_p\otimes dx^j|_p,$$
	are smooth functions,
	\item the index of $g_p$ is the same for all $p\in M$.
\end{enumerate}
In other words, a pseudo-metric is a field of tensors of type  $(0,2)$ that is symmetric, 
non-degenerate  and of constant index. The pair $(M,g)$, where $g$ is a  pseudo-metric on $M$, 
is called a {\bf pseudo-Riemannian manifold}.

The common index $\nu$ of  $g_p$ in a  pseudo-Riemannian manifold $(M,g)$ is the {\bf index} of $M$. 
When $\nu=0$ we say that $(M,g)$ is a {\bf Riemannian manifold}.
In such case $g_p$ determines an inner product over $T_p M$ for all $p\in M$. 
On the other hand, when $\nu=1$ and $n\geq 2$ the pair $(M,g)$ is called a {\bf Lorentzian manifold}. 
In the first case, the signature of $g$ is $(0,n)$ while in the second case $(1,n)$. 
A bilinear form  over a finite dimensional real vector space that satisfies 
the first two conditions of our definition is called a {\bf scalar product}. 
An inner product is a scalar product that is positive definite.

A linear connection $\nabla$ on a pseudo-Riemannian manifold $(M,g)$ 
is said to be {\bf compatible with the  pseudo-metric structure} of $M$ if it satisfies $\nabla g=0$, that is, if 
\begin{equation}\label{metricconnection}
X\cdot g(Y,Z)=g(\nabla_XY,Z)+g(Y,\nabla_XZ),
\end{equation}
for all $X,Y,Z\in\mathfrak{X}(M)$. The following result is usually called the {\bf Fundamental theorem of pseudo-Riemannian Geometry}.
\begin{theorem}[Levi-Civita]\label{Levi-Civitaconnection}
	Given $(M,g)$ a pseudo-Riemannian manifold, 
	there exists a unique linear connection $\nabla$ on $M$ that is compatible with  
	the pseudo-metric structure of  $M$ and has vanishing torsion tensor. 
	Such a linear connection is characterized by the {\bf Koszul formula}
	\begin{eqnarray*}
		2g(\nabla_XY,Z) & =  X\cdot g(Y,Z)+Y\cdot g(Z,X)-Z\cdot g(X,Y)+\\
		& \hskip 10pt  -g(X,[Y,Z])+g(Y,[Z,X])+g(Z,[X,Y]),
	\end{eqnarray*}
	for all $X,Y,Z\in\mathfrak{X}(M)$.
\end{theorem}
The linear connection of  Theorem \ref{Levi-Civitaconnection} is called the {\bf Levi-Civita connection}. 
It is important to observe that the Koszul formula implies that the   
Christoffel symbols associated to the  Levi-Civita connection satisfy the relation
\begin{equation}
\sum_{l=1}^ng_{lk}\Gamma_{ji}^{l}=\dfrac{1}{2}\left(\dfrac{\partial g_{ki}}{\partial x^j}+\dfrac{\partial g_{jk}}{\partial x^i}-\dfrac{\partial g_{ji}}{\partial x^k} \right),
\end{equation}
for all $i,j,k=1,\cdots,n$.
When the curvature tensor of the Levi-Civita connection 
$\nabla$ associated to  a pseudo-Riemannian manifold $(M,g)$ vanishes, 
the pseudo-metric $g$ is called {\bf flat}, and the pair 
$(M,g)$ is a {\bf flat pseudo-Riemannian manifold}. 

The basic model of  flat pseudo-Riemannian manifolds is the space 
$(\mathbb{R}^n_\nu,g_0^\nu,\nabla^0)$ where $\mathbb{R}^n_\nu$ 
equals $\mathbb{R}^n$  with pseudo-metric $g_0^\nu$ of index $\nu$ 
with $0\leq\nu\leq n$, defined by 
$$g_0^\nu=-\sum_{j=1}^\nu dx^j\otimes dx^j+\sum_{j=\nu+1}^n dx^j\otimes dx^j.$$

A simple computation shows that the usual  linear connection  
$\nabla^0$ of $\mathbb{R}^n$ is the  Levi-Civita connection 
associated to $g_0^\nu$. When $\nu=0$, the pseudo-Riemannian manifold 
$\mathbb{R}_\nu^n$ reduces to $\mathbb{R}^n$. 
On the other hand, for $\nu=1$ and $n\geq 2$, the manifold $\mathbb{R}_1^n$ 
is known as  the {\bf $n$-dimensional Minkowski space}. 
The Lorentzian manifold  $(\mathbb{R}_1^4,g_0^1)$ is the basic model for relativistic  space-time.

An {\bf isometry} between two pseudo-Riemannian manifolds 
$(M_1,g_1)$ and $(M_2,g_ 2)$ is a diffeomorphism $f: M_1\to M_2$ satisfying $f^\ast g_2=g_1$, that is,
$$(g_2)_{F(p)}(F_{\ast,p}(X_p),F_{\ast,p}(Y_p))=(g_1)_p(X_p,Y_p),$$
for all $X_p,Y_p\in T_pM_1$ with $p\in M_1$. 

\begin{remark}
	If $(M,g)$ is a  pseudo-Riemannian manifold and $f:M\to M$ is an isometry, 
	the uniqueness of the Levi-Civita connection $\nabla$ associated to  
	$g$ implies that $f$ is an affine transformation of $(M,\nabla)$. 
	More precisely, we have
	\begin{equation}\label{AffineLevi-Civita}
	f_\ast^{-1}\left(\nabla_{f_\ast X}f_\ast Y\right)=\nabla_XY,
	\end{equation}
	for all $X,Y\in\mathfrak{X}(M)$ (see \cite[p.\thinspace161]{KN}).
\end{remark}
If $O(n,\mathbb{R})$ denotes the linear orthogonal group, 
the {\bf group of isometries} of $(\mathbb{R}^n,g_0)$ is the Lie group 
$\text{OAff}(\mathbb{R}^n)=\mathbb{R}^n\rtimes_{Id} O(n,\mathbb{R})$ 
determined by the  semi-direct product of the Abelian  Lie group  
$(\mathbb{R}^n,+)$ and the orthogonal group $O(n,\mathbb{R})$ via the identity representation.  
An important consequence of Theorem   \ref{Characterization1} to the case of Riemannian  manifolds, 
which can be proven in a similar way, is the following result (see for instance \cite{M}).
\begin{proposition}
	A real smooth manifold $M$ of dimension $n$ admits a flat 
	Riemannian metric if and only if there exists an atlas 
	$\lbrace (U_\alpha,\varphi_\alpha) \rbrace_{\alpha\in J}$ of $M$ 
	for which the changes of coordinates are restrictions of the elements of the group of isometries of 
	$(\mathbb{R}^n,g_0)$; that is, for each $\alpha,\beta\in J$ with  
	$U_\alpha\cap U_\beta\neq \emptyset$, there exists 
	$\sigma_{\alpha\beta}\in\text{OAff}(\mathbb{R}^n)$ such that
	$$\left.\varphi_\beta\circ \varphi_\alpha^{-1}\right|_{\varphi_\alpha(U_\alpha\cap U_\beta)}= \left.\sigma_{\alpha\beta}\right|_{\varphi_\alpha(U_\alpha\cap U_\beta)}.$$ \hfill $\square$
\end{proposition}

For the next examples, we denote by $G_j$ the discrete subgroup of  $\text{OAff}(\mathbb{R}^2)$ 
which acts freely and properly discontinuously  over $\mathbb{R}^2$, for $j=1,2,3,4$. 
Recall that in such a case the quotient manifold $\mathbb{R}^2/G_j$ 
admits an affine structure whose changes of coordinates are restrictions of elements of  
$G_j$ (see \cite[p.\thinspace 349]{R}). 
There exists four types  of flat complete $2$-dimensional  Riemannian manifolds other than 
$(\mathbb{R}^2,g_0)$ they are given in the following example. 
See for instance \cite[p.\thinspace 209-224]{KN} for further details.

\begin{example} 
	\emph{Ordinary cylinder.}  Let $G_1$ be the set of transformations of  
	$\mathbb{R}^2$ defined by  $$C_n(x,y)=(x+n,y),\quad\text{for all}\quad n\in\mathbb{Z}.$$
	The quotient manifold $\mathbb{R}^2/G_1$ determined by the action of  $G_1$ over 
	$\mathbb{R}^2$ is diffeomorphic to the ordinary cylinder $S^1\times \mathbb{R}$.
\end{example}
\begin{example}
	\emph{Ordinary torus.} Consider the set $G_2$ of transformations of  
	$\mathbb{R}^2$ given by  
	$$T_{n,m}^{a,b}(x,y)=(x+ma+n,y+mb),$$
	for all $n,m\in\mathbb{Z}$ and $a,b\in\mathbb{R}$, $b\neq 0$.
	The quotient manifold $\mathbb{R}^2/G_2$ determined by the action of  $G_2$ over 
	$\mathbb{R}^2$ is diffeomorphic to the ordinary torus.
\end{example}
\begin{example}
	\emph{Infinite M\"obius band.} We denote by $G_3$ the set of transformations of  $\mathbb{R}^2$ defined by
	$$M_n(x,y)=(x+n,(-1)^ny),$$
	for all $n\in\mathbb{Z}$. 
	The quotient manifold $\mathbb{R}^2/G_3$ determined by the action of  $G_3$ over 
	$\mathbb{R}^2$ is diffeomorphic to 
	the infinite M\"obius band.
\end{example}
\begin{example}
	\emph{Klein bottle.} Let $G_4$ be the set of transformations of  $\mathbb{R}^2$ given by
	$$K_{n,m}^b(x,y)=(x+n,(-1)^ny+bm),$$
	for all $n,m\in\mathbb{Z}$ and $b\in\mathbb{R}\backslash\lbrace 0\rbrace$. 
	The quotient manifold $\mathbb{R}^2/G_4$ determined by the action of  $G_4$ over 
	$\mathbb{R}^2$ is diffeomorphic to the Klein bottle.
\end{example}

The existence of partitions of unity for  $M$ helps us guaranty the existence of  
Riemannian metrics on  $M$. Nevertheless, partitions of unity do not allow us to prove 
the existence of pseudo-metrics on $M$ with index at least $1$. 
In fact, there are topological obstructions to the existence of such pseudo-metrics. 
For example, a compact manifold $M$ admits a  Lorentzian
metric if and only if its Euler characteristic $\chi(M)$ is equal to  zero. 
This because in such cases
we can  guaranty the existence  of a nowhere vanishing vector field on 
$M$ (see \cite{Ma} or \cite[p.\thinspace 207]{S}). 
The only compact two dimensional surfaces satisfying this condition are the torus and the Klein bottle.

\section{Flat pseudo-Riemannian Lie groups}
Let $G$ be a  real connected Lie group of dimension $n$ and $\mathfrak{g}$ its  Lie algebra. 
The goal of this section is to discuss the open problem proposed by  J. Milnor in \cite{M} of describing   
left invariant flat affine structures in the case when  $G$ admits left invariant flat pseudo-metrics.

A pseudo-metric $\mu$ on $G$ is called {\bf left invariant} if $L_\sigma^\ast\mu=\mu$ for all $\sigma\in G$. 
In other words, $\mu$ is left invariant if $L_\sigma$ is an isometry of $(G,\mu)$ for all $\sigma\in G$.
The pair $(G,\mu)$, where $\mu$ is a  left invariant pseudo-metric on $G$, is called a {\bf pseudo-Riemannian Lie group}.

There is a faithful correspondence between  left invariant pseudo-metrics  on 
$G$ and scalar products on $\mathfrak{g}$, 
depicted as follows. If $\mu$ is a  left invariant pseudo-metric on $G$, then 
$\mu_\epsilon:\mathfrak{g}\times\mathfrak{g}\to \mathbb{R}$ 
defines a scalar product on $\mathfrak{g}$. 
On the other hand, given a scalar product 
$\mu_0:\mathfrak{g}\times\mathfrak{g}\to \mathbb{R}$ for $\mathfrak{g}$, 
as a consequence of the chain rule, we can define a left invariant pseudo-metric $\mu$ on $G$ 
by the formula
\begin{equation}\label{pseudometricaSG}
\mu_\sigma(X_\sigma,Y_\sigma)=\mu_0((L_{\sigma^{-1}})_{\ast,\sigma}(X_\sigma),(L_{\sigma^{-1}})_{\ast,\sigma}(Y_\sigma)),
\end{equation}
for all $X_\sigma,Y_\sigma\in T_\sigma G$ with $\sigma\in G$.
If $(G,\mu)$ is a pseudo-Riemannian Lie group, identity \eqref{AffineLevi-Civita} 
implies that the Levi-Civita connection  $\nabla$ associated to  
$\mu$ is a left invariant linear connection. 
On the other hand, since $\mu$ is a  left invariant pseudo-metric we get
$$\mu_\sigma(x^+_\sigma,y^+_\sigma)=\mu_\sigma((L_\sigma)_{\ast,\epsilon}(x),(L_\sigma)_{\ast,\epsilon}(y))=\mu_\epsilon(x,y),$$
for all $x,y\in\mathfrak{g}$.
This implies that the map  $\mu(x^+,y^+):G\to\mathbb{R}$ 
defined by $\sigma\mapsto \mu_\sigma(x^+_\sigma,y^+_\sigma)$, 
is constant for all $x^+,y^+\in\mathfrak{X}_{l}(G)$. 
Therefore, putting $x\cdot y=L_x(y)=(\nabla_{x^+}y^+)(\epsilon)$ 
for $x,y\in \mathfrak{g}$, we have
\begin{equation}\label{productLevi-Civita1}
[x,y]=x\cdot y-y\cdot x\quad\text{and}
\end{equation}
\begin{equation}\label{productLevi-Civita2}
\mu_\epsilon(L_x(y),z)+\mu_\epsilon(y,L_x(z))=0,
\end{equation}
for $x,y,z\in\mathfrak{g}$.
The bilinear map $\cdot:\mathfrak{g}\times\mathfrak{g}\to\mathfrak{g}$ 
defined by $x\cdot y=L_x(y)=(\nabla_{x^+}y^+)(\epsilon)$ is called the 
{\bf Levi-Civita product}. 
The Koszul formula implies that the Levi-Civita product is characterized by the expression
\begin{equation}\label{Levi-CivitaProduct}
\mu_\epsilon(L_x(y),z)=\dfrac{1}{2}(\mu_\epsilon([x,y],z)-\mu_\epsilon([y,z],x)+\mu_\epsilon([z,x],y)),
\end{equation}
where $x,y,z\in\mathfrak{g}$.

A pseudo-Riemannian Lie group $(G,\mu)$ is called {\bf flat} 
if the curvature tensor of the  Levi-Civita connection associated to $\mu$ is identically zero. 

Let $(V,\mu_0)$ be a real finite-dimensional vector space with a scalar product  $\mu_0$. 
The group of {\bf orthogonal transformations} of $(V,\mu_0)$, denoted by $O(V,\mu_0)$, 
is defined as the set of 
transformations $T:V\to V$ which satisfy $\mu_0(T(x),T(y))=\mu_0(x,y)$ for all $x,y\in V$. 
It is a Lie group whose Lie algebra is the set $\mathfrak{o}(V,\mu_0)$ of endomorphisms 
$t:V\to V$ verifying the identity $\mu_0(t(x),y)+\mu_0(x,t(y))=0$ for all $x,y\in V$. 
The {\bf group of isometries} of $(V,\mu_0)$, denoted by $\text{OAff}(V)$, 
is defined as the semi-direct product $V\rtimes_{Id} O(V,\mu_0)$ of the Abelian  
Lie group $(V,+)$ and $O(V,\mu_0)$ via the identity representation.
\begin{remark}
	
	If $(G,\mu)$ is a flat pseudo-Riemannian Lie group  and $\nabla$ is the Levi-Civita 
	connection associated to  $\mu$, the map 
	$L\colon\mathfrak{g}\to\mathfrak{o}(\mathfrak{g},\mu_\epsilon)$ defined by 
	$x\mapsto L_x$, where $L_x\colon \mathfrak{g}\to \mathfrak{g}$ is the linear map  
	defined by $L_x(y)=(\nabla_{x^+}y^+)(\epsilon)$ for all $x,y\in\mathfrak{g}$, 
	is a well defined  Lie algebra homomorphism.
\end{remark}
We can now have a first characterization of  flat pseudo-Riemannian Lie groups as given by  
A. Aubert and A. Medina in \cite{AuM}.
\begin{proposition}[Aubert-Medina]\label{ThmAuberMedina1}
	Let $G$ be a  real connected Lie group of dimension  $n$, $\mathfrak{g}$ its Lie algebra, and 
	$\widetilde{G}$ its universal covering Lie group. Then, the following are equivalent.
	\begin{enumerate}
		\item There exists a left invariant flat pseudo-metric on $G$.
		\item There exist a scalar product $\mu_0:\mathfrak{g}\times\mathfrak{g}\to \mathbb{R}$ 
		and a bilinear map $\cdot:\mathfrak{g}\times\mathfrak{g}\to \mathfrak{g}$ over $\mathfrak{g}$ 
		such that \eqref{productLevi-Civita1} and \eqref{productLevi-Civita2} 
		are satisfied together with 
		$L_{[x,y]}=[L_x,L_y]_{\mathfrak{gl}(\mathfrak{g})}$, for all $x,y\in\mathfrak{g}$.
		\item There exist a real $n$-dimensional vector space $(V,\mu_0)$ 
		together with a scalar product $\mu_0$ and a
		Lie group homomorphism $\rho: \widetilde{G}\to \text{OAff}(V,\mu_0)$ 
		such that the left action of  $\widetilde{G}$ over $V$ 
		defined by $\sigma\cdot v=\rho(\sigma)(v)$ for all $(\sigma,v)\in \widetilde{G}\times V$ 
		admits a point with open orbit and discrete isotropy.
	\end{enumerate}
\end{proposition}

\begin{proof}
	We first prove that $1$ implies $2$. If $(G,\mu)$ is a flat pseudo -
	Riemannian 
	Lie group and 
	$\nabla$ is the Levi-Civita connection associated to  $\mu$, 
	from our previous observations we know that $\mu_\epsilon$ and the Levi-Civita product associated to $\nabla$ 
	satisfy the required identities. 
	
	To see $2$ implies $3$ recall the proof of Theorem \ref{characterizatioFALIC}. 
	By hypothesis, the map $\theta\colon \mathfrak{g}\to \mathfrak{g}\rtimes_{id}\mathfrak{o}(\mathfrak{g},\mu_0)$ 
	defined by $x\mapsto(x,L_x)$ for all $x\in\mathfrak{g}$
	is a well defined Lie algebra homomorphism. Therefore, 
	using the exponential map of $G$, we obtain a homomorphism of  
	Lie groups $\rho\colon \widetilde{G} \to \text{OAff}(\mathfrak{g},\mu_0)$ for which  
	$0\in\mathfrak{g}$ is  a point with open orbit and discrete isotropy  
	for the left action of  $\widetilde{G}$ over $\mathfrak{g}$ defined by 
	$\sigma\cdot x=\rho(\sigma)(x)$ for all $(\sigma,x)\in \widetilde{G}\times \mathfrak{g}$. 
	
	Finally $3$ implies $1$. Let $\rho\colon \widetilde{G}\to \text{OAff}(V,\mu_0)$ 
	be a homomorphism of  Lie groups, defined by $\sigma\mapsto (Q(\sigma),F_\sigma)$ 
	for all $\sigma\in\widetilde{G}$, where $(V,\mu_0)$ is a real vector space of dimension 
	$n$ together with a scalar product $\mu_0$ such that the orbital map  
	$\pi\colon\widetilde{G}\to \text{Orb}(v)$ defined by 
	$\sigma\mapsto Q(\sigma)+F_\sigma(v)$ is a local diffeomorphism for some  
	$v\in V$. Differentiating on the identity of  $\widetilde{G}$, 
	we obtain a Lie algebra homomorphism 
	$\theta\colon \mathfrak{g} \to \mathfrak{g}\rtimes_{id}\mathfrak{o}(V,\mu_0)$ given by  
	$x \mapsto (q(x),f_x)$, where the linear map  $\psi_v\colon \mathfrak{g} \to V$ defined by 
	$x \mapsto q(x)+f_x(v)$ is an isomorphism. 
	Now, define on $\mathfrak{g}$ the scalar product 
	$\widetilde{\mu}_0$ and the bilinear map $\cdot$ respectively by
	$$\widetilde{\mu_0}(x,y)=\mu_0(\psi_v(x),\psi_v(y))$$
	and
	$$L_x=\psi_v^{-1}\circ f_x\circ \psi_v,\quad y\mapsto x\cdot y=L_x(y),$$
	for all $x,y\in\mathfrak{g}$.
	If $\mu$ denotes the left invariant pseudo-metric on $G$ induced by  
	$\widetilde{\mu_0}$ through Formula \eqref{pseudometricaSG}, 
	it is easy to check that  $\cdot$ is the  Levi-Civita product associated 
	to the  Levi-Civita connection determined by  $\mu$, 
	given that we have $f_x\in\mathfrak{o}(V,\mu_0)$ for all $x\in\mathfrak{g}$. 
	Therefore, as in the proof of Theorem \ref{characterizatioFALIC} 
	we conclude that $\mu$ is a left invariant flat  pseudo-metric on $G$.
\end{proof}

The following result allows us to determine  when the  
Levi-Civita connection associated to  a left invariant flat pseudo-metric is geodesically complete. 
Recall that a linear connection $\nabla$ over a smooth manifold  
$M$ is {\bf geodesically complete} if for  any initial condition $(p,X_p)$ $\in TM$ 
its geodesics are defined for all $t\in\mathbb{R}$. 
If $(G,\nabla)$ is a flat  affine Lie group, J. Helmstetter showed in  \cite{H} that 
$\nabla$ is geodesically complete if and only if  
$\text{tr}(R_x)=0$ for all $x\in\mathfrak{g}$, here $R_x\colon \mathfrak{g}\to\mathfrak{g}$ 
is the linear map  defined by $R_x(y)=(\nabla_{y^+}x^+)(\epsilon)$ for all $x,y\in\mathfrak{g}$. 
On the other hand, a Lie group $G$ is called {\bf unimodular} 
if its left invariant  Haar measure is also  right invariant. J. Milnor showed in  
\cite{M1} that a Lie group $G$ is unimodular if and only if 
$\text{det}(\text{Ad}_\sigma)=\pm 1$ for all $\sigma\in G$. 
If $G$ is connected, this is equivalent to requiring $\text{tr}(\text{ad}_x)=0$ 
for all $x\in\mathfrak{g}$ (compare \cite{M1}). For the next result see \cite{AuM}.
\begin{theorem}[Aubert-Medina]\label{unimodularcompletes}
	Let $(G,\mu)$ be a connected flat pseudo-Riemannian Lie group. 
	Then the Levi-Civita connection associated to  $\mu$ is geodesically complete if and only if  $G$ is  unimodular.
\end{theorem}
\begin{proof}
	Let $\nabla$ be the Levi-Civita connection associated to the left 
	invariant flat pseudo-metric $\mu$. We denote by 
	$L_x(y)=x\cdot y=(\nabla_{x^+}y^+)(\epsilon)$ the Levi-Civita product on 
	$\mathfrak{g}$ associated to  $\nabla$. If $\mu_\epsilon$ is  the scalar 
	product on $\mathfrak{g}$ induced by $\mu$, we have
	$$\mu_\epsilon(L_x(y),z)+\mu_\epsilon(y,L_x(z))=0,$$
	for all $x,y,z\in\mathfrak{g}$. This implies that $L_x$ is antisymmetric with  
	respect to  $\mu_\epsilon$. Therefore, if $L_x^\ast\colon\mathfrak{g}\to \mathfrak{g}$ 
	denotes the adjoint operator of  $L_x$ with  respect to  
	$\mu_\epsilon$, we have $L_x^\ast=-L_x$ for all $x\in \mathfrak{g}$. 
	On the other hand, if  $\mathfrak{g}^\ast$ denotes the dual space associated to  
	$\mathfrak{g}$ and  $^tL_x\colon\mathfrak{g}^\ast\to \mathfrak{g}^\ast$ 
	is the transposed of the linear map $L_x$, the linear isomorphism  
	$\varphi\colon\mathfrak{g}\to\mathfrak{g}^\ast$ defined by 
	$\varphi(x)=\mu_\epsilon(x,\cdot)$ where $\varphi(x)(y)=\mu_\epsilon(x,y)$ for all $x,y\in\mathfrak{g}$, 
	fits into the  following commutative diagram for all $x\in \mathfrak{g}$
	$$\xymatrix{
		\mathfrak{g} \ar[d]_{\varphi}\ar[r]^{L_x^\ast} & \mathfrak{g} \ar[d]^{\varphi}\\
		\mathfrak{g}^\ast \ar[r]_{^tL_x} & \mathfrak{g}^\ast; 
	}$$
	thus, we have $L_x^\ast=\varphi^{-1} \circ ^tL_x\circ \varphi$ and so
	$$-\text{tr}(L_x)=\text{tr}(L_x^\ast)=\text{tr}(\varphi^{-1} \circ ^tL_x\circ \varphi)=\text{tr}(^tL_x)=\text{tr}(L_x),$$
	for all $x\in\mathfrak{g}$.
	This implies $\text{tr}(L_x)=0$ for $x\in \mathfrak{g}$. 
	
	Now suppose that $\nabla$ is geodesically complete. 
	Since this is a left invariant flat affine connection, we have $\text{tr}(R_x)=0$, 
	where $R_x\colon\mathfrak{g}\to\mathfrak{g}$ is the linear map  
	defined by $R_x(y)=y\cdot x$ for all $x,y\in \mathfrak{g}$. 
	On the other hand, identity \eqref{productLevi-Civita1} 
	implies $\text{ad}_x=L_x-R_x$, and therefore we get
	$$\text{tr}(\text{ad}_x)=\text{tr}(L_x-R_x)=\text{tr}(L_x)-\text{tr}(R_x)=0,$$
	for all $x\in\mathfrak{g}$, which shows that  $G$ is a unimodular  Lie group.
	
	Reciprocally, if $G$ is a unimodular  Lie group, we have 
	$\text{tr}(\text{ad}_x)=0$ for all $x\in \mathfrak{g}$. 
	As we have $\text{ad}_x=L_x-R_x$ and $\text{tr}(L_x)=0$, 
	we obtain $\text{tr}(R_x)=0$ for all $x\in \mathfrak{g}$. 
	From the fact that $\nabla$ is a left invariant flat affine connection and $\text{tr}(R_x)=0$, 
	we get that it is geodesically complete.
\end{proof}
\begin{example} The group of affine transformations of the line $\text{Aff}(\mathbb{R})$ 
	has a natural left invariant flat Lorentzian metric given by 
	$\mu=\dfrac{1}{x^2}(dx\otimes dy+dy\otimes dx)$. 
	The  Levi-Civita connection associated to  $\mu$ is determined by the rules
	$$\nabla_{e_1^+}e_1^+=- e_1^+, \quad\nabla_{e_1^+}e_2^+=e_2^+,\quad\nabla_{e_2^+}e_1^+=\nabla_{e_2^+}e_2^+=0.$$
	Since $\text{Aff}(\mathbb{R})$  is not unimodular, we have that $\nabla$ 
	is not geodesically complete. The natural left invariant 
	Riemannian metric on $\text{Aff}(\mathbb{R})$ given by 
	$\overline{\mu}=\dfrac{1}{x^2}(dx\otimes dx+dy\otimes dy)$ is also not flat. 
	The Levi-Civita connection associated to $\overline{\mu}$ is determined by 
	$$\overline{\nabla}_{e_1^+}e_1^+=\overline{\nabla}_{e_1^+}e_2^+=0,\quad\overline{\nabla}_{e_2^+}e_1^+=-e_2^+,\quad\overline{\nabla}_{e_2^+}e_2^+=e_1^+.$$
	It is easy to verify that the curvature tensor of $\overline{\nabla}$ is not identically zero. 
	As a consequence of Theorem \ref{MilnorflatTheorem} (of next section) 
	it is possible to show that there does not exist left invariant flat Riemannian metrics on $\text{Aff}(\mathbb{R})$.
\end{example}
\begin{example}
	Over the  Heisenberg group $H_3$, we can  define a left invariant flat Lorentzian metric by $$\mu=dx\otimes dz+dy\otimes dy+dz\otimes dx-x(dx\otimes dy+dy\otimes dx).$$
	The  Levi-Civita connection associated to  $\mu$ is determined by  
	$$\nabla_{e_1^+}=\nabla_{e_3^+}=0, \quad\nabla_{e_2^+}e_1^+=-e_3,\quad\nabla_{e_2^+}e_2^+=e_1^+,\quad\nabla_{e_2^+}e_3^+=0.$$
	Since $H_3$ is unimodular, we have that $\nabla$ is geodesically complete. 
	If $H_{2n+1}$ denotes the Heisenberg group of dimension $2n+1$ for $n\in\mathbb{N}$, then $H_{2n+1}$ 
	is a flat pseudo-Riemannian Lie group if and only if $n=1$ (see \cite{AuM}).
\end{example}

\section{Classical cotangent pseudo-Riemannian Lie group}

A simple construction that allows us to obtain flat pseudo-Riemannian 
Lie groups starting out with connected flat affine Lie  groups  is the following 
(see \cite{AuM}). 

Let $(G,\nabla)$ be a connected affine flat
Lie group of dimension $n$, $\mathfrak{g}$ its Lie algebra, and 
$\widetilde{G}$ its universal covering Lie group. 
Since $\nabla$ is a left invariant flat affine connection,
the map $L\colon \mathfrak{g}\to\mathfrak{gl}(\mathfrak{g})$ defined by 
$x\mapsto L_x$, where $L_x(y)=x\cdot y=(\nabla_{x^+}y^+)(\epsilon)$ for all 
$x,y\in\mathfrak{g}$, is a Lie algebra homomorphism. 
The dual representation associated to   $L$ is the Lie algebra homomorphism 
$L^\ast\colon \mathfrak{g}\to \mathfrak{gl}(\mathfrak{g}^*)$, defined by 
$x\mapsto L^*_x=-^tL_x$, where $L_x^\ast(\alpha)=-\alpha\circ L_x$ 
for all $\alpha\in\mathfrak{g}^\ast$. Using the exponential map of  $G$, 
we obtain a  Lie group homomorphism $\Phi\colon \widetilde{G} \to \text{GL}(\mathfrak{g}^*)$ 
via $\displaystyle \exp_G(x) \mapsto\Phi(\exp_G(x))=\sum_{k=0}^{\infty}\dfrac{1}{k!}(L_x^\ast)^k$, 
namely
$$\Phi_{\ast,\widetilde{\epsilon}}(x)=\left.\dfrac{d}{dt} \right\vert_{t=0}(\Phi(\exp(tx)))=L^*_x,$$
for all $x\in\text{Lie}(\widetilde{G})=\mathfrak{g}$.
Therefore, the product manifold  $T^*\widetilde{G}=\widetilde{G}\times \mathfrak{g}^*$ 
is endowed with the structure of a Lie group given by the semidirect product of  
$\widetilde{G}$ with the Abelian  Lie group  $(\mathfrak{g}^\ast,+)$ through $\Phi$; 
more precisely we have 
$$(\sigma,\alpha)\cdot(\tau,\beta)=(\sigma\tau,\Phi(\sigma)(\beta)+\alpha),$$
for all $\sigma,\tau\in G$ and $\alpha,\beta\in\mathfrak{g}^\ast$. The Lie group $T^*\widetilde{G}=\widetilde{G}\ltimes_\Phi\mathfrak{g}^*$ is 
called the {\bf classical pseudo-Riemannian cotangent Lie group} associated to the flat affine connected Lie group  $(G,\nabla)$.\\

Here the term `classical' stands in contrast to the more general 
construction of twisted cotangent Lie groups as used by A. Aubert and A. Medina (see \cite{AuM}). 
The Lie group $T^*\widetilde{G}$ is then characterized by the following result.

\begin{proposition}[Aubert-Medina]\label{AuMcotangent}
	The Lie algebra of $T^*\widetilde{G}=\widetilde{G}\ltimes_\Phi\mathfrak{g}^*$ 
	is the product vector space  $\mathfrak{g}\ltimes_L\mathfrak{g}^\ast$ with  Lie bracket
	\begin{equation}\label{Hess1}
	[(x,\alpha),(y,\beta)]=([x,y],L^*_x(\beta)-L^*_y(\alpha)),
	\end{equation}
	for all $x,y\in\mathfrak{g}$ and $\alpha,\beta\in\mathfrak{g}^\ast$.
	Moreover, 
	\begin{equation}\label{Hess2}
	\widetilde{\omega}((x,\alpha),(y,\alpha))=\alpha(y)+\beta(x),
	\end{equation}
	for all $x,y\in\mathfrak{g}$ and $\alpha,\beta\in\mathfrak{g}^\ast$, is a scalar product over 
	$\mathfrak{g}\ltimes_L\mathfrak{g}^\ast$ with signature $(n,n)$ 
	which, by formula \eqref{pseudometricaSG}, defines a left invariant flat pseudo-metric   
	$T^*\widetilde{G}$ whose Levi-Civita connection is determined by 
	$$\nabla_{(x,\alpha)^+}(y,\alpha)^+=(x\cdot y,L^\ast_x(\beta))^+,$$
	for all $x,y\in\mathfrak{g}$ and $\alpha,\beta\in\mathfrak{g}^\ast$.
\end{proposition}
\begin{proof}
	As the Lie group structure of $T^*\widetilde{G}$ is given by a semidirect product, it is simple to check that its Lie algebra is the vector space $\mathfrak{g}\ltimes_L\mathfrak{g}^\ast$ with Lie bracket given by \eqref{Hess1}. On the other hand, as the Levi-Civita connection is unique, the proof of the last statement is an immediate consequence of Proposition \ref{ThmAuberMedina1}.
\end{proof}
A more general construction appears in the study of the twisted pseudo-Riemannian cotangent Lie group of a connected flat affine Lie group (see \cite[Proposition 2.1]{AuM}).
\begin{remark}
	If $G$ is simply connected flat affine Lie group, then $T^*\widetilde{G}$ is a trivial vector bundle 
	isomorphic to the cotangent bundle $T^\ast G$ of $G$. It is well known that there is a natural way to 
	associate a structure of Lie group to $T^\ast G$ given that it is isomorphic 
	to the trivial bundle $G\times \mathfrak{g}^\ast$ through the vector bundle isomorphism
	\begin{align*}
	&& T^\ast G &\to G\times \mathfrak{g}^\ast\\
	&&   (\sigma,\alpha_\sigma) &\longmapsto (\sigma, \alpha_\sigma\circ (L_\sigma)_{\ast,\epsilon}).
	\end{align*}
	If $\text{Ad}^\ast\colon  G\to \text{GL}(\mathfrak{g}^\ast)$ 
	denotes the co-adjoint representation of $G$, then the product manifold 
	$G\times \mathfrak{g}^\ast$ has a  Lie group structure given by the semi-direct product of  
	$G$ with  the Abelian  Lie group  $\mathfrak{g}^\ast$ through  $\text{Ad}^\ast$; consequently we have
	\begin{equation}\label{Liecotangent0}
	(\sigma,\alpha)\cdot (\tau,\beta)=(\sigma\tau,\text{Ad}^\ast_\sigma(\beta)+\alpha),
	\end{equation}
	for all $\sigma,\tau\in G$ and $\alpha,\beta\in \mathfrak{g}^\ast$.
	Therefore, $T^\ast G$ has the structure of a  Lie group induced by 
	\eqref{Liecotangent0}. If $\text{ad}^\ast\colon  \mathfrak{g}\to \mathfrak{gl}(\mathfrak{g}^\ast)$ 
	denotes the co-adjoint representation of $\mathfrak{g}$, 
	the Lie algebra of $T^\ast G$ is the product  vector space 
	$\mathfrak{g}\ltimes_{ad^\ast} \mathfrak{g}^\ast$ with   Lie bracket
	\begin{equation}\label{Liealgebracotangent0}
	[(x,\alpha),(y,\beta)]=([x,y]_\mathfrak{g},\text{ad}_x^\ast(\beta)-\text{ad}_y^\ast(\alpha)),
	\end{equation}
	for all $x,y\in\mathfrak{g}$ and $\alpha,\beta\in\mathfrak{g}^\ast$.
	As $G$ is a simply connected Lie group, then it is elementary to verify that 
	$T^\ast \widetilde{G}$ is locally isomorphic to the cotangent bundle $T^\ast G$ as 
	Lie groups if the maps 
	$L^\ast\colon  \mathfrak{g}\to \mathfrak{gl}(\mathfrak{g}^\ast)$ and  
	$\text{ad}^\ast\colon  \mathfrak{g}\to \mathfrak{gl}(\mathfrak{g}^\ast)$ are isomorphic representations, 
	that is, there exists a linear isomorphism 
	$\psi:\mathfrak{g}^\ast\to \mathfrak{g}^\ast$ such that $\text{ad}_x^\ast\circ\psi=\psi\circ L_x^\ast$ 
	for all $x\in \mathfrak{g}$. In this case, the linear map $l:\mathfrak{g}\ltimes_L\mathfrak{g}^\ast\to\mathfrak{g}\ltimes_{\text{ad}^\ast}\mathfrak{g}^\ast$ 
	defined by $(x,\alpha)\mapsto (x,\psi(\alpha))$ is a Lie algebra isomorphism.
\end{remark}

\section{Orthogonal Lie groups}
In this section we will study some elementary properties of those 
Lie groups which have bi-invariant  pseudo-metrics. 
These will be called  {\bf orthogonal Lie groups} (see the definition a few lines down). 
To study the main characteristic of orthogonal Lie groups we introduce the notion of  
{\bf orthogonal Lie algebra} which we will 
be used in the method of {\bf double orthogonal extension}  
described by  A. Medina and Ph. Revoy (see \cite{MR}).

We describe how to construct the oscillator    
Lie algebra of the oscillator Lie group which appears in various branches of 
Physics and Mathematical Physics and give rise to particular solutions of the Einstein-Yang-Mills 
equations (see \cite{L}). Finally, we will provide another characterization of  flat   
Riemannian Lie groups due to  J. Milnor (compare \cite{M1}).

For each $\sigma\in G$, we denote by $R_\sigma\colon G\to G$ the 
right multiplications by  $\sigma$ in $G$, 
which is defined by $R_\sigma(\tau)=\tau\sigma$ for all $\tau\in G$. A pseudo-metric $\mu$ on $G$ is {\bf right invariant} if $R_\sigma^\ast\mu=\mu$ for all $\sigma\in G$. In other words, $\mu$ is right invariant if $R_\sigma$ is an isometry of $(G,\mu)$ for all $\sigma\in G$.
A pseudo-metric $\mu$ on $G$ is called {\bf bi-invariant} if it is left invariant and right invariant. The pair $(G,\mu)$, where $\mu$ is a  bi-invariant pseudo-metric over $G$, is called an {\bf orthogonal Lie group}.

If $\mu$ is a  left invariant pseudo-metric on $G$, it is easy to show that $\mu$  
is right invariant if and only if
\begin{equation}\label{orthogonal1}
\mu_\epsilon(\text{Ad}_\sigma(x),\text{Ad}_\sigma(y))=\mu_\epsilon(x,y)
\end{equation}
holds for all $\sigma\in G$ and $x,y\in\mathfrak{g}$.
This implies in the context that $\mu$ is right invariant if and only if the adjoint representation of  
$\mathfrak{g}$ is antisymmetric with  respect to
$\mu_\epsilon$. The latter implies $\text{ad}_x\in\mathfrak{o}(\mathfrak{g},\mu_\epsilon)$ 
for all $x\in\mathfrak{g}$, namely
\begin{equation}\label{orthogonal2}
\mu_\epsilon([x,y],z)+\mu_\epsilon(y,[x,z])=0,
\end{equation}
for all $x,y,z\in\mathfrak{g}$.

A scalar product over $\mathfrak{g}$ which satisfies identity  
\eqref{orthogonal2} is called an {\bf invariant scalar product}. 
A pair  $(\mathfrak{g},\mu_0)$, where $\mathfrak{g}$ 
is a finite dimensional real Lie algebra and   $\mu_0$ is an invariant scalar product over 
$\mathfrak{g}$ is named an {\bf orthogonal Lie algebra}.

If $\mu_0$ is an invariant scalar product over $\mathfrak{g}$, 
the left invariant pseudo-metric defined by the formula 
\eqref{pseudometricaSG}  is also  right invariant. On the other hand, if $(G,\mu)$ is an orthogonal Lie group, 
the  Koszul formula reduced at the  identity \eqref{Levi-CivitaProduct} 
and expression  \eqref{orthogonal2} imply that the  Levi-Civita 
connection $\nabla$ associated to  $\mu$ is determined by
\begin{equation}\label{orthogonal3}
\nabla_{x^+}y^+=\dfrac{1}{2}[x,y]^+,
\end{equation}
for $x,y\in\mathfrak{g}$.
Moreover, as a consequence of the  Jacobi identity in  
$\mathfrak{g}$, it follows that the curvature tensor of  $\nabla$ 
is given by the expression 
\begin{equation}\label{orthogonal4}
R_\nabla(x^+,y^+)z^+=-\dfrac{1}{4}[[x,y],z]^+,
\end{equation}
with $x,y,z\in\mathfrak{g}$.
\begin{remark}
	A left invariant linear connection on $G$ is called a {\bf Cartan 0-connection} 
	if for all $x\in\mathfrak{g}$, the 1-parameter subgroups of  
	$G$ and the geodesic curves of  $\nabla$ determined by the initial condition  
	$(\epsilon,x)\in G\times\mathfrak{g}$ coincide. 
	It is easy to see that every Cartan  0-connection  is geodesically complete. 
	Moreover, there exists a unique 
	Cartan 0-connection on $G$ with vanishing  torsion, 
	as it is  completely determined by Equation \eqref{orthogonal3} (see \cite[p.\thinspace 72]{P}).
\end{remark}
As an immediate consequence of Identity \eqref{orthogonal4} we have the following result (see for instance \cite{AuM}).
\begin{proposition}[Aubert-Medina]
	Let $(G,\mu)$ be an orthogonal Lie group. 
	The bi-invariant  pseudo-metric $\mu$ is flat if and only if $G$ is a  2-nilpotent
	Lie group.
	\hfill$\square$
\end{proposition}

\begin{example} \emph{Semisimple  Lie groups.} Let $G$ be a semisimple Lie group. 
	It is well known that $G$ is a semisimple if and only if
	the \emph{Killing form} of $\mathfrak{g}$, 
	which we denote by $k\colon \mathfrak{g}\times\mathfrak{g}\to \mathbb{R}$,  
	defined by $(x,y)\mapsto k(x,y)=\text{tr}(\text{ad}_x\circ \text{ad}_y)$ 
	for $x,y\in \mathfrak{g}$, is non-degenerate. 
	Direct computation shows  
	$$k([x,y],z)=-k(y,[x,z]),$$
	for $ x,y,z\in \mathfrak{g}$. Therefore $(\mathfrak{g},k)$ is an orthogonal Lie algebra.
\end{example}
\begin{example}
	\emph{The cotangent bundle of a Lie group.} Let $G$ be a real connected $n$-dimensional Lie group, 
	$\mathfrak{g}$ its Lie algebra, $T^\ast G$ the cotangent bundle of $G$, 
	and $\mathfrak{g}^\ast$ the dual vector space of $\mathfrak{g}$. 
	Recall that $T^\ast G$ is isomorphic to the trivial bundle $G\times \mathfrak{g}^\ast$ 
	and this is endowed with a natural Lie group structure given by
	\begin{equation}\label{Liecotangent}
	(\sigma,\alpha)\cdot (\tau,\beta)=(\sigma\tau,\text{Ad}^\ast_\sigma(\beta)+\alpha),
	\end{equation}
	for all $\sigma,\tau\in G$ and $\alpha,\beta\in \mathfrak{g}^\ast$.
	Therefore, the Lie algebra of $T^\ast G$ is the product vector space 
	$\mathfrak{g}\ltimes_{ad^\ast} \mathfrak{g}^\ast$ with   Lie bracket
	\begin{equation}\label{Liealgebracotangent}
	[(x,\alpha),(y,\beta)]=([x,y]_\mathfrak{g},\text{ad}_x^\ast(\beta)-\text{ad}_y^\ast(\alpha)),
	\end{equation}
	for all $x,y\in\mathfrak{g}$ and $\alpha,\beta\in\mathfrak{g}^\ast$.
	
	We define  over $\mathfrak{g}\ltimes_{ad^\ast} \mathfrak{g}^\ast$ the function
	\begin{equation}\label{invariantcotangent}
	\mu_0((x,\alpha),(y,\beta))=\alpha(y)+\beta(x),
	\end{equation}
	for all $x,y\in\mathfrak{g}$ and $\alpha,\beta\in\mathfrak{g}^\ast$.
	It is easy to see that $\mu_0$ defines an invariant scalar product over 
	$\mathfrak{g}\ltimes_{ad^\ast} \mathfrak{g}^\ast$, of  signature $(n,n)$, so that
	$(\mathfrak{g}\ltimes_{ad^\ast} \mathfrak{g}^\ast,\mu_0)$ is an orthogonal Lie algebra.
\end{example}

\begin{example}
	\emph{Oscillator Lie group.} For $\lambda=(\lambda_1,\cdots,\lambda_n)\in\mathbb{R}^n$ 
	with  $0<\lambda_1\leq\cdots\leq\lambda_n$, the  $\lambda$-oscillator Lie group, 
	denoted by $G_\lambda$, is determined by the product  manifold  
	$\mathbb{R}^{2n+2}\cong \mathbb{R}\times\mathbb{R}\times \mathbb{C}^n$ endowed with the product
	$$(t,s,z_1,\cdots,z_n)\cdot(t',s',z'_1,\cdots,z'_n)=$$
	$$= \left(t+t',s+s'+\dfrac{1}{2}\sum_{j=1}^n \text{Im}(\overline{z_j}z'_je^{i\lambda_jt}),z_1+z'_1e^{i\lambda_1t},\cdots,z_n+z'_ne^{i\lambda_nt}\right),$$
	where $t,t',s,s'\in\mathbb{R}$ y $z_j,z'_j\in\mathbb{C}$ for all $j=1,\cdots, n$. 
	The  Lie algebra of  $G_\lambda$, denoted by  $\mathfrak{g}_\lambda$, 
	is isomorphic to the vector space 
	$\mathbb{R}e\times \mathbb{R}^{2n}\times\mathbb{R}\hat{e}=\text{Vect}_\mathbb{R}\lbrace e,e_j,\hat{e_j},\hat{e} \rbrace_{j=1,\cdots, n}$
	, with Lie bracket
	$$[e,e_j]=\lambda_j\hat{e_j},\hspace{0.5cm} [e,\hat{e_j}]=-\lambda_j e_j,\hspace{0.5cm} [e_j,\hat{e_j}]=\hat{e},$$
	for all $j=1,\cdots n$.
	
	If $\displaystyle x=\alpha e+\sum_{j=1}^nx_je_j+\sum_{j=1}^ny_j\hat{e_j}+\beta\hat{e}$ 
	denotes an element of $\mathfrak{g}_\lambda$, the function  $\mu_0$ defined over 
	$\mathfrak{g}_\lambda\times\mathfrak{g}_\lambda$ by
	$$\mu_0(x,y)=\sum_{j=1}^n\dfrac{1}{\lambda_j}(x_jx'_j+y_jy'_j)+\alpha\beta'+\alpha'\beta$$
	is an invariant scalar product over $\mathfrak{g}_\lambda$. This allow us to conclude that  
	$G_\lambda$ is an orthogonal Lie group. 
	The signature of the scalar product $\mu_0$ is $(1,2n+1)$ so that it determines,  
	by means of Formula  \eqref{pseudometricaSG},
	a bi-invariant  Lorentzian metric over $G_\lambda$.
	\begin{remark}
		The  $\lambda$-oscillators Lie groups are the only solvable simply connected 
		non-Abelian Lie  groups that  admit a bi-invariant Lorentzian metric (see \cite{Me2}).
		The oscillator 4-dimensional Lie group  has its origin in the study of the 
		harmonic oscillator which is one of the simplest non-relativistic  
		systems where the  Schr\"odinger equation can be completely solved. 
		Moreover, oscillator Lie groups  are particular  solutions to the Einstein-Yang-Mills  equations (see \cite{L}).
		Over oscillator Lie groups there exist infinitely many solutions to the Yang-Baxter equations (see \cite{BM}).
	\end{remark}
\end{example} 

\begin{example}
	\emph{ A non-orthogonal Lie group: $\text{Aff}(\mathbb{R})$.} 
	The Lie group of affine transformations of the line $\text{Aff}(\mathbb{R})$ 
	is a classical example of a non-orthogonal Lie Group. 
	If there were an invariant scalar product $\mu_0$ over $\mathfrak{aff}(\mathbb{R})$, we will get
	$$\mu_0([x,y],z)+\mu_0(y,[x,z])=0,$$
	for all $x,y,z\in\mathfrak{g}$. If we replace here $x=e_1$, $y=e_2$ and $z=e_1$, we obtain $\mu_0(e_1,e_2)=0$. 
	On the other hand, if we replace  $x=e_1$, $y=e_2$ and $z=e_2$ we get $\mu_0(e_2,e_2)=0$. 
	Therefore, the element $e_2$ is orthogonal with respect to $\mu_0$ to all
	elements of  $\mathfrak{aff}(\mathbb{R})$, which contradicts the fact that  $\mu_0$ is non-degenerate. 
	The bottom line is that there does not exist an invariant scalar product over $\mathfrak{aff}(\mathbb{R})$.
\end{example}

If $G$ is a compact Lie group, using the  Haar measure on  
$G$ we can construct a bi-invariant  Riemannian metric 
over $G$ (see \cite[p.\thinspace 340]{P}). 
In further generality, a connected Lie group  $G$ admits a bi-invariant  metric  
if and only if  it is isomorphic to the Cartesian product of  a compact group and an 
additive vector group (see \cite{M1}). 
On the other hand, if $(\mathfrak{g},\mu_0)$ is an orthogonal Lie algebra with   
an invariant inner product $\mu_0$ and  $\mathfrak{h}$ is an ideal of  $\mathfrak{g}$ 
for each $y$ in the orthogonal  complement $\mathfrak{h}^{\perp_{\mu_0}}$ of 
$\mathfrak{h}$ with  respect to  $\mu_0$, we have
$$\mu_0([x,y],h)=-\mu_0(y,[x,h])=0,$$
for $x\in\mathfrak{g}$ and $h\in\mathfrak{h}$. 
This implies that $\mathfrak{h}^{\perp_{\mu_0}}$ is also an ideal of $\mathfrak{g}$. 
Therefore, by induction, have shown that  $\mathfrak{g}$ 
can be expressed as an orthogonal direct sum of simple  ideals (see \cite{M1}).

\begin{remark}
	K. Iwasawa showed in \cite{I} that  if $G$ is a connected Lie group, 
	then every compact subgroup is contained in a maximal compact subgroup $H$, 
	which is also connected. Moreover, topologically $G$ is isomorphic to the 
	Cartesian product  of $H$ with  an Euclidean space $\mathbb{R}^k$.
	For $(G,\mu)$ a flat Riemannian  Lie group,
	if we ignore for a moment the group structure of $G$ and think of it just as a  
	Riemannian manifold, we have that $G$ is isometric to Euclidean space. 
	Therefore, as a consequence of Iwasawa's theorem, every compact subgroup of $(G,\mu)$ is commutative (see \cite{M1}).
\end{remark}

The following characterization of flat  Riemannian Lie groups is due to  J. Milnor (see \cite{M1}).
\begin{theorem}[Milnor]\label{MilnorflatTheorem}
	Let $(G,\mu)$ be a  Riemannian Lie group. 
	The metric $\mu$ is flat if and only if  the Lie algebra $\mathfrak{g}$ 
	decomposes as an orthogonal  direct sum $\mathfrak{b}\oplus\mathfrak{u}$, 
	where $\mathfrak{b}$ is an Abelian subalgebra and
	$\mathfrak{u}$ is an Abelian ideal such that the linear maps $\text{ad}_b$ are 
	antisymmetric with respect to  $\mu_\epsilon$ for all $b\in\mathfrak{b}$. 
\end{theorem}
\begin{proof}
	Suppose that $(G,\mu)$ is a flat  Riemannian Lie group. 
	If $\nabla$ is the Levi-Civita connection associated to  $\mu$, then by 
	Proposition \ref{ThmAuberMedina1}, we know that the linear map 
	$L\colon \mathfrak{g}\to\mathfrak{o}(\mathfrak{g},\mu_\epsilon)$ 
	defined by $x\mapsto L_x$, where $L_x(y)=x\cdot y=(\nabla_{x^+}y^+)(\epsilon)$ 
	for all $x,y\in\mathfrak{g}$, is a well defined  Lie algebra homomorphism. 
	We denote by  $\mathfrak{u}$ the kernel of $L$. 
	Clearly $\mathfrak{u}$ is an ideal of $\mathfrak{g}$. 
	Since the torsion tensor of  $\nabla$ vanishes, we have $[x,y]=L_x(y)-L_y(x)$ for all 
	$x,y\in\mathfrak{g}$. In particular, for $u,v\in\mathfrak{u}$ we have $[u,v]=0$ 
	and it follows that $\mathfrak{u}$ is an Abelian ideal. Let $\mathfrak{b}$ be the orthogonal 
	complement of  $\mathfrak{u}$ with  respect to $\mu_\epsilon$. 
	For each $b\in\mathfrak{b}$ we have the identity
	$$\text{ad}_b(u)=[b,u]=L_b(u)-L_u(b)=L_b(u),$$
	for $u\in\mathfrak{u}$. Given that  $\mathfrak{u}$ is an ideal of $\mathfrak{g}$, 
	the linear map  $L_b$ takes $\mathfrak{u}$ onto itself. 
	Therefore, $L_b$ takes the orthogonal complement  $\mathfrak{b}$ to itself, 
	and  since this is true  for all $b\in\mathfrak{b}$, 
	we conclude that $\mathfrak{b}$ is a  Lie  subalgebra of $\mathfrak{g}$. 
	On the other hand, since $L$ is a Lie algebra homomorphism and 
	$\mathfrak{u}=\text{Ker}(L)$, we have that $\mathfrak{b}$ is sent isomorphically  
	to a Lie subalgebra $L(\mathfrak{b})$ of $\mathfrak{o}(\mathfrak{g},\mu_\epsilon)$. 
	For simplicity, we  denote  $L(\mathfrak{b})$ also  by $\mathfrak{b}$. 
	Given that  $\mathfrak{o}(\mathfrak{g},\mu_\epsilon)$ is the Lie algebra of the  
	compact Lie group $O(\mathfrak{g},\mu_\epsilon)$, which admits a bi-invariant  
	Riemannian metric, we deduce the existence of   an invariant inner product  
	$\mu_0$ over $\mathfrak{o}(\mathfrak{g},\mu_\epsilon)$. 
	Since $\mathfrak{b}$ is a  Lie subalgebra of $\mathfrak{o}(\mathfrak{g},\mu_\epsilon)$, 
	it is easy to verify that $\mu_0$ restricts  naturally  to an invariant inner  product over 
	$\mathfrak{b}$. Therefore, $\mathfrak{b}$ can be written as  an  orthogonal direct sum  
	$\mathfrak{b}_1\oplus \cdots\oplus \mathfrak{b}_k$ of  simple ideals. 
	If any of these  simple ideal, say $\mathfrak{b}_j$, were non Abelian, 
	then the  corresponding simple Lie group  
	$G_j$ must be compact (see \cite[Thm 2.2]{M1}) and the inclusion 
	$\mathfrak{b}_j\subset\mathfrak{b}\subset\mathfrak{g}$ 
	would imply the existence of a nontrivial Lie group homomorphism  $G_j\to G$. 
	Hence, $G$ must contain a non-trivial compact subgroup,  which is a  contradiction. 
	Therefore, each  $\mathfrak{b}_j$ must be Abelian and accordingly  
	$\mathfrak{b}$ is an Abelian Lie subalgebra.
	Finally, since for each $b\in\mathfrak{b}$ the restriction of  
	$\text{ad}_b$ to $\mathfrak{b}$ is the  trivial map, whereas
	we have $\text{ad}_b=L_b$ when restricting $\text{ad}_b$ to $\mathfrak{u}$, 
	we obtain $\text{ad}_b\in\mathfrak{o}(\mathfrak{g},\mu_\epsilon)$ for all $b\in\mathfrak{b}$.
	
	Reciprocally, suppose that the Lie algebra $\mathfrak{g}$  decomposes as an   
	orthogonal direct sum $\mathfrak{b}\oplus\mathfrak{u}$, 
	where $\mathfrak{b}$ is an Abelian subalgebra and  $\mathfrak{u}$ is an 
	Abelian ideal such that  $\text{ad}_b\in\mathfrak{o}(\mathfrak{g},\mu_\epsilon)$ for all $b\in\mathfrak{b}$.
	As $\mu_\epsilon$ is nondegenerate, the Koszul formula reduced to the identity \eqref{Levi-CivitaProduct} and
	both formulas \eqref{productLevi-Civita1} and \eqref{productLevi-Civita1}
	imply that the Levi-Civita product associated to $\mu_\epsilon$ satisfies the identities
	$$L_u=0,\qquad L_b=\text{ad}_b,$$ 
	for all $u\in\mathfrak{u}$ and $b\in\mathfrak{b}$.
	It is easy to verify that this implies  $L_{[x,y]}=[L_x,L_y]_{\mathfrak{gl}(\mathfrak{g})}$ 
	for all $x,y\in\mathfrak{g}$. Therefore, by  Proposition \ref{ThmAuberMedina1} 
	we have that $\mu$ is a left-invariant flat Riemannian metric. 
\end{proof}

\begin{example}
	\emph{$\text{Aff}(\mathbb{R})$ does not admit a left-invariant flat Riemannian metric.} 
	Recall that the Lie  Algebra of
	$\text{Aff}(\mathbb{R})$ is $\mathfrak{aff}(\mathbb{R})=\text{Vect}_\mathbb{R}\lbrace e_1,e_2\rbrace$ 
	with Lie bracket $[e_1,e_2]=e_2$. 
	Suppose that $\text{Aff}(\mathbb{R})$ admits a left-invariant flat Riemannian metric $\mu$. 
	Let  $\nabla$ be the Levi-Civita 
	connection associated to  $\mu$ and  $L_x(y)=(\nabla_{x^+}y^+)(\epsilon)$ the  
	Levi-Civita product determined by  $\nabla$. By Theorem \ref{MilnorflatTheorem} 
	$\mathfrak{aff}(\mathbb{R})$ decomposes as an orthogonal direct sum   
	$\mathfrak{b}\oplus \mathfrak{u}$ where  $\mathfrak{u}=\text{Ker}(L)$ is an Abelian ideal of  
	$\mathfrak{aff}(\mathbb{R})$ and $\mathfrak{b}$ is an Abelian subalgebra of 
	$\mathfrak{aff}(\mathbb{R})$ such that  
	$\text{ad}_b\in\mathfrak{o}(\mathfrak{aff}(\mathbb{R}),\mu_\epsilon)$ for all  
	$b\in\mathfrak{b}$. Given these conditions, it is clear that we have  
	$\mathfrak{u}=\mathbb{R}e_2$ and $\mathfrak{b}=\mathbb{R}e_1$. 
	Therefore, from $\text{ad}_{e_1}\in\mathfrak{o}(\mathfrak{aff}(\mathbb{R}),\mu_\epsilon)$ 
	we get $\mu_\epsilon(e_1,e_2)=0$, and since 
	$$e_2=[e_1,e_2]=L_{e_1}(e_2)-L_{e_2}(e_1)=L_{e_1}(e_2),$$
	the condition $L_{e_1}\in\mathfrak{o}(\mathfrak{aff}(\mathbb{R}),\mu_\epsilon)$ 
	implies $\mu_\epsilon(e_2,e_2)=0$. 
	Consequently,   $e_2$ is  orthogonal to every element of  
	$\mathfrak{aff}(\mathbb{R})$ with respect to  $\mu_\epsilon$, which  
	contradicts the fact that  $\mu_\epsilon$ is non-degenerate.
\end{example}

\section{The double  orthogonal extension}
In what follows we describe a construction method  
known by the name of  {\bf double orthogonal extension} which is due to   
A. Medina and Ph. Revoy (compare \cite{MR}). This method  provides, among other things, 
a way to construct all finite-dimensional orthogonal Lie  algebras. 
As an application of the double orthogonal extension we indicate how to construct  the 
Lie algebra of the   $\lambda$-oscillator Lie group.

Given an  orthogonal Lie algebra $(\mathfrak{g},\mu_0)$, the space of  
{\bf skew - symmetric derivations} of  $\mathfrak{g}$ with  respect to  $\mu_0$, 
denoted by  $\text{Der}_a(\mathfrak{g},\mu)$, is defined as the set of 
derivations $D\colon \mathfrak{g}\to\mathfrak{g}$ that verify 
$\mu_0(D(x),y)=-\mu_0(x,D(y))$ for all $x,y\in\mathfrak{g}$. 
It is easy to check that  $\text{Der}_a(\mathfrak{g},\mu)$ is a  
Lie subalgebra of  $\text{Der}(\mathfrak{g})$. 
Suppose that there exists a Lie algebra homomorphism 
$\psi\colon \mathfrak{h}\to \text{Der}_a(\mathfrak{g},\mu)$ for some  Lie algebra 
$\mathfrak{h}$. Define  the map 
$\Phi\colon \mathfrak{g}\times\mathfrak{g} \to \mathfrak{h}^\ast$ by $\Phi(x,y)(z)=\mu_0(\psi_z(x),y)$ 
for $x,y\in \mathfrak{g}$ and $z\in\mathfrak{h}$. 
Such a map is clearly  bilinear. 
Moreover, since $\psi_z\in \text{Der}_a(\mathfrak{g},\mu)$ for all $z\in\mathfrak{h}$, 
we have that $\Phi$ is skew-symmetric and satisfies
\begin{equation}\label{orthogonal5}
\Phi([x,y],w)+\Phi([y,w],x)+\Phi([w,x],y)=0,
\end{equation}
for all $x,y,w\in\mathfrak{g}$.

The properties of  $\Phi$ together with identity \eqref{orthogonal5} 
tell us that $\Phi$ defines a $2$-cocycle of the Lie algebra $\mathfrak{g}$ with 
values in the vector space $\mathfrak{h}^\ast$ with  respect to the trivial representation of  
$\mathfrak{g}$ by $\mathfrak{h}^\ast$ (see \cite{CE}). 
Therefore, the product vector space $\mathfrak{g}_\Phi=\mathfrak{g}\times_{\Phi}\mathfrak{h}^\ast$ 
is a Lie algebra with Lie bracket given by 
$$[(x,\alpha),(y,\beta)]_c=([x,y]_{\mathfrak{g}},\Phi(x,y)),$$
for all $x,y\in\mathfrak{g}$ and $\alpha,\beta\in\mathfrak{h}^\ast$.
In what follows we denote by 
$\pi^\ast\colon \mathfrak{h}\to \mathfrak{gl}(\mathfrak{h}^\ast)$ 
the co-adjoint representation of $\mathfrak{h}$. 
For each $z\in\mathfrak{h}$, we define the map 
$\Theta_z\colon \mathfrak{g}\times_{\Phi}\mathfrak{h}^\ast\to \mathfrak{g}\times_{\Phi}\mathfrak{h}^\ast$ by 
$(x,\alpha) \mapsto (\psi_z(x),\pi_z^\ast(\alpha))$ 
for all $x\in\mathfrak{g}$ and $\alpha\in \mathfrak{h}^\ast$. 
Since $\psi\colon \mathfrak{h}\to \text{Der}_a(\mathfrak{g},\mu)$ is a Lie algebra homomorphism, it  is easy to verify they satisfy 
\begin{equation}\label{orthogonal6}
\pi_z^\ast(\Phi(x,y))=\Phi(\psi_z(x),y)+\Phi(x,\psi_z(y)),
\end{equation}
for all $x,y\in\mathfrak{g}$ and $z\in\mathfrak{h}$.
Formula \eqref{orthogonal6} aids us to show that $\Theta_z$ 
is a derivation of the  Lie algebra $(\mathfrak{g}\times_{\Phi}\mathfrak{h}^\ast,[\cdot ,\cdot ]_c)$ 
for each $z\in\mathfrak{h}$, namely, we get
$$\Theta_z([(x,\alpha),(y,\beta)]_c)=[\Theta_z(x,\alpha),(y,\beta)]_c+[(x,\alpha),\Theta_z(y,\beta)]_c,$$
for all $x,y\in\mathfrak{g}$ and $\alpha,\beta\in\mathfrak{h}^\ast$. 
Therefore, the map 
$\Theta\colon \mathfrak{h} \to \text{Der}(\mathfrak{g}\times_{\Phi}\mathfrak{h}^\ast,[\cdot ,\cdot]_c)$, 
defined by $z\mapsto \Theta_z$, is a well behaved Lie algebra homomorphism. 
The vector space 
$\widetilde{\mathfrak{g}}\colon =\mathfrak{h}\ltimes_\Theta(\mathfrak{g}\times_{\Phi}\mathfrak{h}^\ast)$ 
has the structure of a Lie algebra given by the  semidirect product of  
$\mathfrak{h}$ with  $\mathfrak{g}\oplus_{\Phi}\mathfrak{h}^\ast$ through the Lie 
algebra homomorphism $\Theta$; in other words, 
the Lie bracket on $\widetilde{\mathfrak{g}}$  is given explicitely by
$$[(z,x,\alpha),(z',y,\beta)]=([z,z']_\mathfrak{h},\psi_z(y)-\psi_{z'}(x)+[x,y]_\mathfrak{g},$$
\begin{equation}
\hskip 2.2in \pi_z^\ast(\beta)-\pi_{z'}^\ast(\alpha)+\Phi(x,y)),
\end{equation}
for all $x,y\in\mathfrak{g}$, $z,z'\in\mathfrak{h}$ and $\alpha,\beta\in\mathfrak{h}^\ast$. 
Finally, over $\widetilde{\mathfrak{g}}\times\widetilde{\mathfrak{g}}$ 
we define the function  $\widetilde{\mu_0}$ as
\begin{equation}\label{produc1}
\widetilde{\mu_0}((z,x,\alpha),(z',y,\beta))=\mu_0(x,y)+\alpha(z')+\beta(z),
\end{equation}
for all $x,y\in\mathfrak{g}$, $z,z'\in\mathfrak{h}$ and $\alpha,\beta\in\mathfrak{h}^\ast$. 
Since $\mu_0$ is an invariant scalar product over $\mathfrak{g}$, 
a direct calculation shows that  $\widetilde{\mu_0}$ 
is an invariant scalar product on $\widetilde{\mathfrak{g}}$ so that  
$(\widetilde{\mathfrak{g}},\widetilde{\mu_0})$ is an orthogonal 
Lie algebra called the {\bf double orthogonal extension} of  
$(\mathfrak{g},\mu_0)$ by  $\mathfrak{h}$ via $\psi$.
\begin{remark}
	If the signature of the  invariant scalar product $\mu_0$ is $(p,q)$, 
	then the signature of  $\widetilde{\mu_0}$ is $(p+\text{dim}(\mathfrak{h}),q+\text{dim}(\mathfrak{h}))$.
\end{remark}
\begin{example}
	\emph{The cotangent bundle of a Lie group.} 
	If in the method of  double orthogonal extension we set
	$\mathfrak{g}=\lbrace 0\rbrace$, it is easy to see that we get 
	$\widetilde{\mathfrak{g}}=\mathfrak{h}\ltimes_{\pi^\ast}\mathfrak{h}^\ast$ with 
	Lie bracket given by  \eqref{Liealgebracotangent} and 
	$\widetilde{\mu_0}((x,\alpha),(y,\beta))=\alpha(y)+\beta(x)$ for all $x,y\in\mathfrak{h}$ 
	and $\alpha,\beta\in\mathfrak{h}^\ast$. 
	Therefore, the orthogonal Lie algebra obtained  is the Lie algebra of the cotangent bundle of the 
	connected and simply connected Lie group  $H$ with  Lie algebra $\mathfrak{h}$.
\end{example}
\begin{example}
	\emph{The Lie  algebra of the  $\lambda$-oscillator Lie group.} 
	Let $\mathfrak{g}=\mathbb{R}^{2n}$ be considered as an Abelian  Lie algebra and 
	$\mu_0=\langle\cdot,\cdot\rangle$ the usual inner product on $\mathbb{R}^{2n}$. 
	Clearly $(\mathbb{R}^{2n},\mu_0)$ is an orthogonal Lie algebra. 
	We define the linear map  
	$\delta\colon \mathbb{R}^{2n}\to \mathbb{R}^{2n}$ by 
	$$x=(x^1,\cdots, x^{2n}) \mapsto (-x^{n+1},\cdots,-x^{2n},x_1,\cdots,x_n)$$ 
	which satisfies
	$$\mu_0(\delta(x),y)=-\sum_{j=1}^nx_{j+n}y_j+\sum_{j=1}^nx_jy_{j+n}=-\mu_0(x,\delta(y)),$$
	for all $x,y\in\mathbb{R}^{2n}$.
	If $\mathfrak{h}=\mathbb{R}e$ is a unidimensional Lie algebra, 
	then the map $\psi\colon \mathbb{R}e \to \text{Der}_a(\mathbb{R}^{2n},\mu_0)$ 
	defined by $te \mapsto \psi(te)=t\delta$, is a well defined Lie algebra homomorphism. 
	Direct calculation shows that  
	$\mathbb{R}^{2n}_\Phi=(\mathbb{R}^{2n}\times_\Phi\mathbb{R}e^\ast,[\cdot,\cdot]_c)$ 
	is isomorphic to the  Heisenberg Lie algebra of  dimension  $2n+1$ and that $\widetilde{\mathbb{R}^{2n}}=\mathbb{R}e\ltimes_\Theta(\mathbb{R}^{2n}\times_\Phi\mathbb{R}e^\ast)$ 
	is the Lie algebra with  bracket
	$$[e,e_j]=\delta(e_j)=\hat{e_j},\qquad [e,\hat{e_j}]=\delta(\hat{e_j})=-e_j,\qquad [e_j,\hat{e_j}]=e^\ast,$$
	for all $j=1,\cdots,n$.
	The invariant product  $\widetilde{\mu_0}$ defined by $\widetilde{\mathbb{R}^{2n}}$ is given by 
	$$\widetilde{\mu_0}(\gamma e+x+\alpha e^\ast,\gamma'e+y+\beta e^\ast)=\mu_0(x,y)+\alpha\gamma'+\beta\gamma,$$
	for all $x,y\in\mathbb{R}^{2n}$ and $\alpha,\beta,\gamma,\gamma'\in\mathbb{R}$. 
	The orthogonal Lie algebra $(\widetilde{\mathbb{R}^{2n}},\widetilde{\mu_0})$ 
	is isomorphic to the   $\lambda$-oscillator Lie algebra with  $\lambda_j=1$ for all $j=1,\cdots,n$. 
	
	A slight modification of this construction allows us to obtain the  
	Lie algebra $\mathfrak{g}_\lambda$ for
	$\lambda=(\lambda_1,\cdots,\lambda_n)\in\mathbb{R}^n$ with  arbitrary $0<\lambda_1\leq\cdots\leq\lambda_n$.
\end{example}

\begin{remark}
	A. Medina and Ph. Revoy  proved in \cite{MR}   that one can inductively produce 
	all orthogonal Lie algebras starting out with simple and unidimensional ones 
	by taking direct sums and double extensions. 
	More precisely, let $\mathfrak{g}$ be an indecomposable orthogonal Lie algebra, 
	that is, an orthogonal Lie algebra that cannot be written as the direct sum of two non-trivial orthogonal Lie algebras. 
	Then either $\mathfrak{g}$ is simple, or $\mathfrak{g}$ is unidimensional, 
	or else $\mathfrak{g}$ is a double extension of an orthogonal Lie algebra 
	$\widetilde{\mathfrak{g}}$ by a unidimensional or a simple Lie algebra $\mathfrak{h}$.  
	As an application, it is possible to show that any indecomposable non-simple 
	Lie algebra of a Lorentzian Lie group with dimension greater than $1$ is the 
	double orthogonal extension of an Abelian Lie algebra with  inner product by a unidimensional Lie algebra. 
	This provides a classification of Lorentzian orthogonal Lie algebras up to isomorphism (see \cite{Me2}).
\end{remark}

\section*{\bf Acknowledgements}
I  am grateful for the support of the Network NT8 from the Office of External Activities of 
Abdus Salam International 
Centre for Theoretical Physics, Italy,
that made  possible my visit to
Universidad Cat\'olica del Norte in Chile. 
I also wish to express my sincere gratitude to
Elizabeth Gasparim who coordinated the realization of the Summer School and 
of this notes. She wrote the English version of these lecture notes.  Finally, I thank Omar Saldarriaga, Alberto Medina, and Alfredo Poirier for their collaboration and valuable comments which facilitated the writing.

 	\end{document}